\documentclass[11pt,leqno]{article}
\usepackage{authblk}
\usepackage[T1]{fontenc}
\usepackage{amsmath}%
\usepackage{amsthm}
\usepackage{amsxtra}%
\usepackage{amsfonts}%
\usepackage{amssymb}%
\usepackage[margin=1.25in]{geometry}
\usepackage{color,hyperref}
\usepackage{graphicx}
\usepackage{subfig}
\usepackage[scientific-notation=true]{siunitx}
\usepackage{cite}
\hypersetup{colorlinks,breaklinks,
             linkcolor=blue,urlcolor=blue,
             anchorcolor=blue,citecolor=blue}
\usepackage{color}
\usepackage{array}
\usepackage{booktabs}
\usepackage{enumerate}
\usepackage{dsfont}
\usepackage{hhline}

\DeclareMathOperator*{\opt}{opt}
\DeclareMathOperator*{\argmax}{argmax}

\newcommand{\one}{\mathds{1}}
\renewcommand{\S}{\mathbb{S}}
\newcommand{\usc}{\text{USC}}
\newcommand{\lsc}{\text{LSC}}

\newcommand{\A}{{\mathcal A}}
\DeclareMathOperator{\Lip}{Lip}
\renewcommand{\O}{{\mathcal O}}
\renewcommand{\bar}[1]{{\overline{#1}}}
\newcommand{\F}{{\mathcal F}}
\newcommand{\E}{{\mathbb E}}

\renewcommand{\H}{{\mathcal H}}
\newcommand{\B}{{\mathcal B}}
\newcommand{\R}{\mathbb{R}}

\newcommand{\Z}{\mathbb{Z}}

\newcommand{\eps}{\varepsilon}

\renewcommand{\phi}{\varphi}

\renewcommand{\v}{\mathbf{v}}
\newcommand{\e}{\mathbf{e}}
\newcommand{\p}{\mathbf{p}}
\newcommand{\w}{\mathbf{w}}
\newcommand{\D}{\mathbb{D}}

\newcommand{\Zh}{\Z^n_h}
\newcommand{\Zn}{\Z^n}
\newcommand{\Zmh}{\Z^{n-1}_h}

\renewcommand{\epsilon}{\varepsilon}

\def\XXint#1#2#3{{\setbox0=\hbox{$#1{#2#3}{\int}$ }
\vcenter{\hbox{$#2#3$ }}\kern-.6\wd0}}

\newtheorem{theorem}{Theorem}
\newtheorem{lemma}[theorem]{Lemma}

\newtheorem{proposition}[theorem]{Proposition}
\theoremstyle{definition}
\newtheorem{remark}[theorem]{Remark}
\newtheorem{definition}[theorem]{Definition}

\newtheorem{conjecture}[theorem]{Conjecture}

\numberwithin{equation}{section}
\numberwithin{theorem}{section}

\title{Numerical solution of a PDE arising from prediction with expert advice\thanks{ {\bf Funding:} Calder and Mosaphir were partially supported by NSF-DMS grant 1944925, and Calder was partially supported by the Alfred P. Sloan foundation, a McKnight Presidential Fellowship, and the Albert and Dorothy Marden Professorship. Drenska was partially supported by NSF-DMS grant 2407839. {\bf Source Code:} \url{https://github.com/jwcalder/PredictionPDE}}}
\author{Jeff Calder}
\affil{School of Mathematics, University of Minnesota\thanks{{\bf Email:} \textit{jcalder@umn.edu}}}
\author{Nadejda Drenska}
\affil{Department of Mathematics, Louisiana State University\thanks{{\bf Email:} \textit{ndrenska@lsu.edu}}}
\author{Drisana Mosaphir}
\affil{School of Mathematics, University of Minnesota\thanks{{\bf Email:} \textit{mosap001@umn.edu}}}
\begin{document} 

\maketitle

\begin{abstract}
This work investigates the online machine learning problem of prediction with expert advice in an adversarial setting through numerical analysis of, and experiments with, a related partial differential equation.  The problem is a repeated two-person game involving decision-making at each step informed by $n$ experts in an adversarial environment. The continuum limit of this game over a large number of steps is a degenerate elliptic equation whose solution encodes the optimal strategies for both players. We develop numerical methods for approximating the solution of this equation in relatively high dimensions ($n\leq 10$) by exploiting symmetries in the equation and the solution to drastically reduce the size of the computational domain. Based on our numerical results we make a number of conjectures about the optimality of various adversarial strategies, in particular about the non-optimality of the COMB strategy.
\end{abstract}

\section{Introduction}

This paper is focused on the classical online learning problem of prediction with expert advice.  Given the advice of $n$ experts who each make predictions in real time about an unknown time-varying quantity (e.g., the price of a stock or option at some time in the future), a player must decide which expert's advice to follow.  The problem is often formulated in an \emph{online} setting, whereby at each step of the game, the player has knowledge of the historical performance of each expert and may use this information to decide which expert to follow at that step. The overall goal is to perform as well as the best performing expert, or as close to this as possible.  We are particularly interested in the adversarial setting, where the performance of the experts is controlled by an adversary, whose goal is to minimize the gains of the player. 

A simple and common way to formulate the game is to assume each expert's prediction is either correct or incorrect at each time step. Given $n$ experts, this can be described by a binary vector $\v\in \B^n$, $\B:=\{0,1\}$, where $v_i=1$ if expert $i$ is correct, and $v_i=0$ otherwise. The player gains $1$ if the expert they chose made a correct prediction, and gains nothing otherwise; thus the player gains $v_i$ if they follow expert $i$. The performance of the player is often measured with the notion of \emph{regret} with respect to each expert, which is the difference between a given expert's gains and the player's gains. We let $x=(x_1,\dots,x_n)\in \R^n$ denote the regret vector with respect to all $n$ experts, so $x_i$ is the regret to the $i^{\rm th}$ expert, or rather, the number of times the $i^{\rm th}$ expert was correct minus the number of times the player was correct. After $T$ steps, the gain of the player and any given expert is at most $T$, and the worst case regret with respect to any given expert is at most $T$. The goal of the player is to minimize their regret with respect to the best performing expert; thus, the player would like to minimize $\max\{x_1,\dots,x_n\}$ at the end of the game. The goal of mathematical and numerical analysis is to compute or approximate the optimal player strategies (i.e., determine which expert the player should follow). There are two standard approaches for how long the game is played; the finite horizon setting, where the game is played for a fixed number of steps $T$, and the geometric horizon setting, where the game ends with probability $\delta>0$ at each step. In the geometric horizon setting, the number of steps of the game is a random variable following the geometric distribution with parameter $\delta$. In this work we focus on the geometric horizon problem, but we expect our techniques to work in the finite horizon setting as well.

In order to proceed further, we need to make a modelling assumption on the experts. In this paper, we follow the convention of a \emph{worst-case} analysis where we assume the experts are controlled by an adversary whose goal is to maximize the player's regret at the end of the game. The adversarial setting yields a two-player zero sum game, and introduces another mathematical problem of determining the optimal strategies for the adversary.  We also focus on the more general setting of \emph{mixed strategies} where the player and adversary both employ randomized strategies. At each step, the player chooses a probability distribution $\alpha$ over the experts $\{1,\dots,n\}$, and the expert followed by the player is drawn independently (from other steps and the adversary's choices) from the distribution $\alpha$. Likewise, the adversary chooses a probability distribution $\beta$ over the binary sequences $\v\in \B^n$, and the experts are advanced by drawing a sample $\v\in\B^n$ from the distribution $\beta$. The problem of determining optimal strategies then boils down to deciding how the player and market should set their probability distributions $\alpha$ and $\beta$ at each time step, given the current regret vector $x\in \R^n$. The goal of the player is to minimize their expected regret, while the adversary's goal is to maximize the expected regret.

It is a difficult problem to determine the optimal strategies for the player and the adversary for a general number of experts $n$ and a large number of steps $T$ in the game. Initial results were established in Cover's original paper \cite{cover1966behavior} for $n=2$ experts, and more recently in \cite{gravin2016towards} for $n=3$ experts. A breakthrough occurred in a series of papers by Drenska and Kohn \cite{drenska2017pde, drenska2020prediction, drenska2020pde}, who took the perspective that the prediction from expert advice problem is a discrete analogue of two-player differential games \cite{friedman2013differential}. They formulated a value function for the game, and showed that as the number of steps $T$ of the game tends to infinity, the rescaled value function converges to the \emph{viscosity solution} $u\in C(\R^n)$ of the degenerate elliptic partial differential equation (PDE)
\begin{equation}\label{eq:pwea}
u(x) - \frac{1}{2}\max_{\v \in \B^n} \v^T\nabla^2u(x)\,\v = \max\{x_1,\dots,x_n\} \ \ \text{for} \ \  x\in\R^n.
\end{equation}
The state variable $x\in \R^n$ of the PDE \eqref{eq:pwea} is the regret vector at the start of the game, and the solution of the PDE $u(x)$ is the worst case regret over all the experts at the end of the game, provided each player plays optimally (and in the limit as $T\to \infty$). Thus, the long-time behavior of the adversarial prediction with expert advice problem, and the corresponding asymptotically optimal strategies for the player and adversary, can be determined by solving a PDE! It turns out that the optimal player strategy is to choose the probability distribution $\alpha(x) = \nabla u(x)$, while the optimal adversarial strategy involves the binary vectors $\v$ saturating the maximum over $\v\in \B^n$ in \eqref{eq:pwea} (we give more detail in Section \ref{sec:background}). 

In this paper, we develop numerical methods for approximating the viscosity solution of \eqref{eq:pwea}, so that we may shed light on the optimal strategies for the player and adversary. There are two challenging aspects of solving \eqref{eq:pwea}. First, the PDE is posed on all of $\R^n$ and does not have any kind of natural restriction to a compact computational domain with boundary conditions. To address this, we prove a localization result for \eqref{eq:pwea} showing that the domain may be restricted to a box $\Omega_T=[-T,T]^n$, and that errors in the Dirichlet boundary condition $u\vert_{\partial \Omega_T}$ \emph{do not} propagate far into the interior of $\Omega$, allowing us to obtain an accurate solution sufficiently interior to $\Omega_T$ (precise results are given later). The second challenge is that the theory is fairly complete for the $n=2,3,4$ expert problems, so the interesting cases to study are in the fairly high dimensional setting of $n\geq 5$ experts, where it is generally difficult to solve PDEs on regular grids or meshes, due to the curse of dimensionality. To overcome this issue, we exploit symmetries in the equation \eqref{eq:pwea}---in particular, permutation invariance of the coordinates $x_1,\dots,x_n$---in order to restrict the computational domain to the sector in $\R^{n-1}$ where $T\geq x_1 \geq x_2 \geq \cdots x_{n-1}\geq x_n = 0$. We develop a numerical method that allows us to restrict all computations to this sector, which has vanishingly small measure compared to the whole box $\Omega_T$. This allows us to numerically solve the PDE \eqref{eq:pwea} on reasonably fine grids up to dimension $n=10$. Based on our numerical results, we make a number of conjectures about optimality of various strategies. We summarize these results in Section \ref{sec:main} after giving  a more thorough description of the background material.

\subsection{Background}\label{sec:background}

Online prediction with expert advice is an example of a sequential decision making problem, and has applications in algorithm boosting \cite{FS}, stock price prediction and portfolio optimization \cite{FS}, self-driving car software \cite{AKT}, and many other problems. The prediction with expert advice problem originated in the works of Cover \cite{cover1966behavior} and Hannan \cite{Hannan}, who provided optimal strategies for the two expert problem. In the intervening years, much attention has been focused on heuristic algorithms that give good performance, but may not be optimal. A commonly used algorithm is the \emph{multiplicative weights algorithm}, in which the player maintains a set of positive weights $w_1,\dots,w_n$ for each expert that are used to form a weighted average of the expert advice, and the weights are updated in real time based on the performance of each expert. For the finite horizon problem, it has been shown \cite{CBL} that the multiplicative weights algorithm is optimal in an asymptotic sense, as the number of experts $n$ and number of steps of the game $T$ \emph{both} tend to infinity, but fails to be optimal when the numbers of experts is fixed and finite, as is the case in practice. Optimal algorithms for the geometric horizon and finite horizon problems for $n \leq 3$ experts were developed and studied in Gravin et al. \cite{gravin2016towards} and Abbasi et al. \cite{ABG}. Additionally, lower bounds for regrets in a broader class than multiplicative weights has been shown \cite{gravin2017tight}.  Further work on algorithms for both for the finite and geometric horizon problems is contained in \cite{ABG, CBL, HKW, LW, CFH, Ro}.  

Recently, attention has shifted back to the problem of optimal strategies for a finite number of experts \cite{gravin2016towards,drenska2017pde,bayraktar2020finite,bayraktar2020malicious,bayraktar2020asymptotic,bayraktar2021prediction,drenska2020pde,drenska2020prediction,drenska2021prediction,calder2021asymp}. The focus of this paper is on the problem with \emph{mixed strategies} (i.e., random strategies) against an adversarial environment that was briefly described in the previous section. We now describe this setting in more detail. We have $n$ experts making predictions, a player who chooses at each step which expert to follow, and an adversarial environment that decides which experts gain or lose at each step of the game. The strategies are \emph{mixed}, so both the player and the adversary choose probability distributions over their possible plays, and their actual plays are random variables drawn from those distributions. At the $k^{\rm th}$ step of the game, the player chooses a probability distribution $\alpha_k$ over the experts $\A_n:=\{1,2,\dots,n\}$ and the adversary chooses a probability distribution $\beta_k$ over the binary sequences $\B^n = \{0,1\}^n$.  Random variables from these distributions $i_k\sim \alpha_k$ and $\v_k\sim \beta_k$ are drawn, and the player chooses to follow expert $i_k$ and the market advances the experts corresponding to the positions of the ones in the binary vector $\v_k$. 

The performance of the player is measured by their regret to each expert, which is the difference between the expert's gains and those of the player. We let $x=(x_1,\dots,x_n)\in \R^n$ denote the regret vector, so $x_i$ is the current regret with respect to expert $i$.  Let us write the coordinates of $\v_k$ as $\v_k=(v_{k,1},v_{k,2},\dots,v_{k,n})$.  Then on the $k^{\rm th}$ step the player accumulates regret of $v_{k,j}-v_{k,i_k}$ with respect to expert $j$. If the regret vector started at $x\in \R^n$ on the first step of the game, then the regret after $T$ steps is 
\[R_T:= x +  \sum_{k=1}^T (\v_k - v_{k,i_k}\one).\]
At the end of the game the regret $R_T$ is evaluated with a payoff function $g:\R^n\to \R$. The player's goal is to minimize $g(R_T)$ while the adversary's goal is to maximize $g(R_T)$. The most commonly used payoff is the maximum regret $g(x) = \max\{x_1,x_2,\dots,x_n\}$, which simply reports the regret of the player with respect to the best performing expert. The game can be played in the finite horizon setting where $T$ is fixed, or the geometric horizon setting where the game ends with probability $\delta>0$ at each step, so $T$ is a random variable. 

We focus on the geometric horizon problem. In this case, the value function $U_\delta:\R^n \to \R$ is defined by
\begin{equation}\label{eq:value}
  U_\delta(x) = \inf_{\alpha}\sup_{\beta}\E\left[ g\left(x +  \sum_{k=1}^T (\v_k - v_{k,i_k}\one) \right)\right],
\end{equation}
where $\alpha=(\alpha_1,\alpha_2,\dots)$, $\beta=(\beta_1,\beta_2,\dots)$, and $T$ is the time at which the game stops, which is a geometric random variable with parameter $\delta$. The value function $U_\delta(x)$ is the expected value of the payoff at the end of the game given that the regret vector starts at $x\in \R^n$ on the first step and both players play optimally. The $\inf$ and the $\sup$ are over \emph{strategies for the players}, which enforce that $\alpha_k$ and $\beta_k$ depend only on the current value of the regret and the past choices of both players. The value function is unchanged by swapping the $\inf$ and $\sup$ in \eqref{eq:value}.

It was shown in \cite{drenska2020prediction,drenska2017pde} that the rescaled value functions
\[u_\delta(x) :=  \sqrt{\delta}\,U_\delta\left( \frac{x}{\delta}\right)\]
converge locally uniformly, as $\delta\to 0$, to the viscosity solution of the degenerate elliptic PDE
\begin{equation}\label{eq:pde}
u -\frac{1}{2}\max_{\v\in B_{n}}\v^T \nabla^2 u \,\v = g \ \ \text{ on } \ \R^n,
\end{equation}
which is the same as \eqref{eq:pwea} except with a general payoff $g$. The PDE \eqref{eq:pde} contains all of the information about the prediction problem in the asymptotic regime where the number of steps $T$ tends to $\infty$, which is equivalent to sending the geometric stopping probability $\delta \to 0$. It was shown in \cite{drenska2020prediction} that the asymptotically optimal player strategy is to use the probability distribution\footnote{As we will see later in the paper, the solution $u$ of \eqref{eq:pde} satisfies $\nabla u \cdot \one = 1$, and is monotonically increasing, so $\alpha_k$ is a probability vector.}
\begin{equation}\label{eq:optimalplayer}
\alpha_k = \nabla u(x),
\end{equation}
where $x\in \R^n$ is the current regret on the $k^{\rm th}$ step of the game, and the optimal adversary strategy is to choose
\begin{equation}\label{eq:optimaladversary}
  \v_k \in \argmax_{\v \in \B^n} \{\v^T \nabla^2 u(x)\v\},
\end{equation}
and then advance the experts in $\v_k$ with probability $\frac{1}{2}$ and those in $\one-\v_k$ with probability $\frac{1}{2}$. Notice that the adversarial strategy $\v$ is equivalent to $\one-\v$. 

Determining the adversary's optimal strategies in the max-regret setting, where we take $g(x)=\max\{x_1,\dots,x_n\}$, is an open problem for a general number of experts $n$. It was conjectured in \cite{gravin2016towards} that the COMB strategy of ordering the experts by regret and using the alternating zeros and ones vector $\v=(0,1,0,1,\dots)$, which resembles the teeth of a comb, is asymptotically optimal as $T\to \infty$ for all $n$, though this was proven in \cite{gravin2016towards} only for $n=2$ and $n=3$ experts. The idea behind the COMB strategy is to group the experts into as equally matched groups as possible and advance one group or the other with equal probability, which makes it difficult for the player to gain any advantage. However, since the experts are ordered by regret $x_1\geq x_2 \geq \cdots \geq x_n$, they are essentially ordered by performance, and so the even numbered experts are slightly worse on average compared to the odd numbered experts. Hence, there is reason to believe that COMB can be improved upon for larger numbers of experts by modifying the COMB vector slightly. 

The PDE perspective developed in \cite{drenska2017pde} can help shed light on the optimal adversarial strategy, since it turns out we can derive \emph{explicit} solutions for the PDE \eqref{eq:pwea} for $n\leq 4$ experts. Below, we present the solutions in the sector
\[\S_n = \{x\in \R^n \, : \, x_1 \geq x_2 \geq \cdots \geq x_n\}.\]
Since the PDE is unchanged under permutations of the coordinates, this completely determines the solution. It was shown in \cite{drenska2017pde} that for $n = 2$ experts the solution of \eqref{eq:pwea} is given in the sector $\S_2$ by
\begin{equation}\label{eq:solution2experts}
u(x) = x_1 + \dfrac{1}{2\sqrt{2}}e^{\sqrt{2}(x_2-x_1)},
\end{equation}
and for $n = 3$, the solution is given in $\S_3$ by
\begin{equation}\label{eq:solution3experts}
u(x) = x_1 + \dfrac{1}{2\sqrt{2}}e^{\sqrt{2}(x_2-x_1)} + \dfrac{1}{6\sqrt{2}}e^{\sqrt{2}(2x_3-x_2-x_1).}
\end{equation}
Since we have explicit solutions, we can check the optimal adversarial strategies via \eqref{eq:optimaladversary} within the sector $\S_n$. For $n=2$ both $\v=(1,0)$ and $\w=(0,1)$ are globally optimal for all $x\in \S_n$, which are both the same COMB strategy since $\v = \one - \w$. For $n=3$, both $(1,0,0)$ and $(0,1,0)$ are globally optimal (as well as a their equivalent strategies $(0,1,1)$ and $(1,0,1)$, which we will omit from now on). The second strategy $(0,1,0)$ is the COMB strategy, while the first $(1,0,0)$ is not.

For $n = 4$ experts, it was shown in \cite{bayraktar2020asymptotic} that the solution of \eqref{eq:pwea} is given by
\begin{equation}\label{eq:solution4experts}
\begin{aligned}
u(x) ={} & x_1 - \dfrac{\sqrt{2}}{4}\text{sinh}(\sqrt{2}(x_1-x_2)) \\
&+ \dfrac{\sqrt{2}}{2}\text{arctan}\left(e^{\frac{x_4+x_3-x_2-x_1}{\sqrt 2}}\right)\cdot\\&\text{cosh}\left(\frac{x_4-x_3+x_2-x_1}{\sqrt 2}\right)\text{cosh}\left(\frac{-x_4+x_3+x_2-x_1}{\sqrt 2}\right)\text{cosh}\left(\frac{-x_4-x_3+x_2+x_1}{\sqrt 2}\right)\\ 
&+ \dfrac{\sqrt{2}}{2}\text{arctanh}\left(e^{\frac{x_4+x_3-x_2-x_1}{\sqrt 2}}\right)\cdot\\&\text{sinh}\left(\dfrac{x_4-x_3+x_2-x_1}{\sqrt 2}\right)\text{sinh}\left(\dfrac{-x_4+x_3+x_2-x_1}{\sqrt 2}\right)\text{sinh}\left(\dfrac{-x_4-x_3+x_2+x_1}{\sqrt 2}\right),
\end{aligned}
\end{equation}
from which it is possible to prove \cite{bayraktar2020asymptotic} that the COMB strategy $(1,0,1,0)$ is optimal, as well as the non-COMB strategy $(0,1,1,0)$. For $n\geq 5$ experts, an explicit solution is unknown, and the question of the optimality of the COMB strategy is open. 

It is worthwhile mentioning that it is remarkable that explicit solutions have been obtained for the degenerate elliptic PDE \eqref{eq:pwea} for $n\leq 4$ dimension. Usually explicit solutions for nonlinear PDE are not available. Furthermore, the solutions given in \eqref{eq:solution2experts}, \eqref{eq:solution3experts}, and \eqref{eq:solution4experts} are all twice continuously differentiable with Lipschitz second partial derivatives, i.e., they are classical $C^{2,1}(\R^n)$ solutions. This is also remarkable, since one would only expect this for \emph{uniformly elliptic} equations (since the right hand side is Lipschitz continuous) and not for degenerate elliptic equations. As far as we are aware, there is no general regularity theory that can explain this.

In fact, the existence of classical solutions of \eqref{eq:pwea} is closely tied to the existence of a globally optimal strategy. As a simple example, for $n=2$ experts the strategy $\v=(1,0)$ is optimal in $\S_2$, and so the PDE \eqref{eq:pwea} reduces to the one dimensional linear equation
\[u - \frac{1}{2}u_{x_1x_1} = x_1 \ \ \text{in} \ \ \S_2.\]
We can integrate this, keeping only the exponentially decaying solution, to obtain
\[u(x) = x_1 + f(x_2) e^{-\sqrt 2 x_1} \ \ \text{on} \ \ \S_2.\]
Using the symmetry $u(x_1,x_2)=u(x_2,x_1)$ we obtain $u_{x_1}=u_{x_2}$ on $\partial \S_2 = \{x\in \R^2 \, : \, x_1=x_2\}$. This yields
\[1 - \sqrt 2 f(x) e^{-\sqrt 2 x} = u_{x_1}(x,x) = u_{x_2}(x,x) = f'(x) e^{-\sqrt 2 x},\]
and so 
\[f(x) = Ce^{-\sqrt 2 x} + \frac{1}{2\sqrt 2}e^{\sqrt 2 x}.\]
Since $u_{x_1}(0)=u_{x_2}(0)=\frac{1}{2}$ we have $f(0) = \frac{1}{2\sqrt 2}$ and so $C=0$. Substituting this above yields the two expert solution $u$ given by \eqref{eq:solution2experts}. Roughly the same procedure can be carried out for $n=3$ and $n=4$ experts, though the $n=4$ case is particularly tedious (see \cite{bayraktar2020asymptotic}).  

The question of whether the COMB strategy is asymptotically optimal for $n\geq 5$ experts is an open problem. Some recent work \cite{chase2019experimental} gives experimental numerical evidence that COMB is \emph{not optimal} for $n=5$ experts. The numerical experiments in \cite{chase2019experimental} simulated the two player game and compared the COMB strategy against the strategy $(1,0,1,0,0)$, the latter appearing to be strictly better. It was shown in \cite{kobzar2020new,kobzar2020newb} that COMB is at least as powerful as the setting of randomly choosing which half of the experts to advance (the so-called Bernoulli strategy). This motivates our work of numerically solving the PDE \eqref{eq:pwea} in order to shed light on the optimality of COMB and other strategies for $n\geq 5$ experts. 

We mention that an analogous parabolic PDE exists for the finite-horizon setting of this problem \cite{drenska2017pde}, motivating a similar treatment in terms of numerics for the parabolic equation as well. Some progress in the parabolic case can be found in \cite{bayraktar2020finite}.  Further, related settings such as prediction against a limited adversary \cite{bayraktar2021prediction} (rather than the optimal adversary being studied in this work) and malicious experts \cite{bayraktar2020malicious} have been studied as well.  A related PDE where $\v$ is chosen just from the standard basis vectors $\{\e_1, \dots, \e_n\}$ has been studied and has a closed-form solution for a general number $n$ of experts \cite{kobzar2020new, kobzar2020newb}.  Additionally, some PDE approaches to online learning problems and neural network-based problems can be found in \cite{wang2021pde}.  We also mention that repeated two-player games also appear in the PDE literature in multiple other settings \cite{KS1,KS2,PS1,PS2,AS1,NS,APSS,LM,calder2020convex,calder2018game,calder2019lip,flores2019algorithms,calder2024consistency}.

\subsection{Main results and conjectures}\label{sec:main}

In Section \ref{sec:numerics} we present the results of numerical experiments solving the prediction with expert advice PDE \eqref{eq:pwea} for $n\leq 10$ experts. We summarize the main results we obtain from the numerical experiments here.  
\begin{enumerate}
\item We have strong numerical evidence that the COMB strategy is not globally optimal for $5 \leq n \leq 10$. This validates the numerical evidence from \cite{chase2019experimental} for $n=5$ experts. 
\item For $n=5$ experts, we have strong numerical evidence that the non-COMB adversarial strategy $(0,1,0,1,1)$, equivalent to $(1,0,1,0,0)$, is the only globally optimal strategy. This is the same strategy that was numerical shown to be better than COMB in \cite{chase2019experimental}.
\item For $6 \leq n \leq 10$ experts, we have strong numerical evidence that there are no globally optimal adversary strategies.
\end{enumerate}

From these numerical results we state a number of conjectures that we leave for future work. 
\begin{conjecture}\label{con:comb}
The COMB strategy is globally asymptotically optimal (i.e., on the sector $\S_n$) only for $n\leq 4$ experts. 
\end{conjecture}
\begin{conjecture}\label{con:5ex}
The non-COMB strategy $(1,0,1,0,0)$ is the only globally asymptotically optimal strategy for $n=5$ experts.
\end{conjecture}
\begin{conjecture}\label{con:6ex}
There is no globally asymptotically optimal adversary strategy that is constant on the sector $\mathbb{S}_n$ for $n\geq 6$ experts. 
\end{conjecture}

Recalling our discussion in Section \ref{sec:background} about the connection between explicit solutions of the PDE \eqref{eq:pwea} and the existence of optimal strategies, if Conjecture \ref{con:5ex} is true, then we expect there to exist a classical explicit solution of the $n=5$ expert PDE, similar to the solutions given in \eqref{eq:solution2experts}, \eqref{eq:solution3experts}, and \eqref{eq:solution4experts} for the $n=2,3,4$ expert problems.\\

\noindent {\bf Open Problem}: Determine an analytic expression for the solution of the PDE \eqref{eq:pwea} for $n=5$.\\

An explicit solution for $n=5$ experts would allow one to check the validity of Conjecture \ref{con:5ex}, as was done for $n=4$ experts in \cite{bayraktar2020asymptotic}. If Conjecture \ref{con:6ex} is true, then this strongly suggests that it will be impossible to find an explicit solution of the PDE for $n\geq 6$ experts, and that the solution may fail to be a classical solution in $C^{2,1}(\R^n)$ when $n\geq 6$. If there is no globally optimal strategy, then there will be regions of $\S_n$ corresponding to different optimal strategies $\v\in \B^n$, and the solution may fail to be twice continuously differentiable across these interfaces. 

\begin{remark}
It is important to point out that there are certainly some limitations to our numerical results. In particular, we cannot solve the PDE \eqref{eq:pwea} numerically on the full space $\R^n$, and must restrict our attention to a compact subset. Our numerical results are obtained over the box $[-1,1]^n$. We cannot rule out, for example, that Conjecture \ref{con:5ex} fails somewhere outside of the box $[-1,1]^n$. However, we do find that for $n=2,3,4$ experts, the optimality observed on the box matches perfectly with the global theory. The negative results of Conjectures \ref{con:comb} and \ref{con:6ex} do not suffer the same limitations, since non-optimality need only be observed at a single point. Finally, our numerical convergence rates rely on classical regularity of the viscosity solution, which is only known for $n\leq 4$ experts. 
\end{remark}

\begin{remark}
We also remark that it may be possible to use our techniques for reducing the computational grid to the sector $\S_n$ in combination with the dynamic programming approaches of \cite{gravin2016towards,chase2019experimental}. We leave this interesting direction to future work. 
\end{remark}

\subsection{Outline}

The rest of the paper is organized as follows. In Sections \ref{sec:numericalschemes} and \ref{sec:predictionanalysis} we present and analyze our numerical scheme for solving the prediction PDE \eqref{eq:pwea}. The first part in Section \ref{sec:numericalschemes} is written for a more general class of degenerate elliptic PDEs, while the second part Section \ref{sec:predictionanalysis} contains results that require the specific form of our prediction with expert advice equation. Finally in Section \ref{sec:numerics} we present the results of our numerical experiments that provide the evidence for the conjectures given in Section \ref{sec:main}. 

\section{Analysis of a general numerical scheme}\label{sec:numericalschemes}

We study here a finite difference scheme for the PDE \eqref{eq:pde} consisting of replacing the pure second derivatives $\v^T \nabla^2 u \v$ with finite difference approximations. We work on the grid $\Z_h^n$, where $\Z_h=h\Z$ and $h>0$ is the grid spacing. For a function $u:\Zh\to \R$, we define the discrete gradient as the mapping $\nabla_h:\Zh \times \Zn\to \R$ defined by
\begin{equation}\label{eq:discrete_grad}
\nabla_h u(x,\v) := \frac{u(x + h\v) - u(x)}{h}.
\end{equation}
The discrete Hessian is defined as the mapping $\nabla^2_h u:\Zh\times\Zn \to \R$ given by
\begin{equation}\label{eq:discrete_hessian}
\nabla^2_h u(x,\v) := \frac{u(x + h\v) - 2u(x) + u(x- h\v)}{h^2}.
\end{equation}
We note that
\begin{equation}\label{eq:hessian_grad}
  \nabla^2_h u(x,\v) = \frac{\nabla_h u(x,\v) + \nabla_h u(x,-\v)}{h}.
\end{equation}
Also, the definition of the discrete Hessian leads immediately to the discrete Taylor-type expansions
\begin{equation}\label{eq:discrete_taylor}
\frac{1}{2}(u(x+h\v) + u(x-h\v)) = u(x) + \frac{h^2}{2}\nabla^2_h u(x,\v),
\end{equation}
and 
\begin{equation}\label{eq:discrete_taylor2}
  u(x+h\v) = u(x) - h\nabla_h(u,-\v) + h^2 \nabla^2_h u(x,\v).
\end{equation}

Define $\H_n = \{X:\Z_n \to \R\}$.
We will use the notation $\nabla_h u(x)\in \H_n$ and $\nabla^2_h u(x)\in \H_n$ for the mappings $\v\mapsto \nabla_h u(x,\v)$ and $\v \mapsto \nabla^2 u_h(x,\v)$, respectively. For $u\in C^{k,1}(\R^n)$ with $k=2$ or $k=3$ and any $\v \in \Zn$ we have via Taylor expansion that
\[  \nabla^2_h u(x,\v) = \v^T \nabla^2 u(x) \v + \O(|\v|^{k+1}h^{k-1}).\]

Our discrete approximation of \eqref{eq:pde} is given by
\[u(x) - \frac{1}{2}\max_{\v \in \B^n}\nabla^2_h u(x,\v) = g(x) \ \ \text{for } x\in \Zh.\]

Since our methods are not specific to this equation, we will study a more general class of equations of the form
\begin{equation}\label{eq:general_discrete_pde}
u - F(\nabla^2_h u) = g \ \ \text{on } \Zh,
\end{equation}
where $F:\H_n\to \R$.  The PDE \eqref{eq:pde} is obtained by setting 
\begin{equation}\label{eq:ourF}
F(X) = \frac{1}{2}\max_{\v\in \B^n}X(\v). 
\end{equation}

We need to place a monotonicity assumption on $F$. 
\begin{definition}\label{def:monotonicity}
We say that $F:\H_n\to \R$ is \emph{monotone}  if for all $X,Y\in \H_n$ with $X\leq Y$ we have $F(X) \leq  F(Y)$.
\end{definition}
We note that $X,Y:\H_n\to \R$ are real-valued functions, so $X\leq Y$ means that $X(\v)\leq Y(\v)$ for all $\v\in \Z_n$.  The class of monotone equations of the form \eqref{eq:general_discrete_pde} is closely related to the wide stencil finite difference schemes introduced and studied by Oberman \cite{oberman2006convergent}. The main difference in this section is that we are focused on the unbounded domain $\Zh$, and we are interested in properties of the solution $u$, such as Lipschitzness, convexity, permutation invariance, etc., that hold under certain structure conditions and the source term $g$.

For $X\in \H_n$ we define
\begin{equation}\label{eq:infty}
\|X\|_{N,\infty}= \max_{\substack{\v\in \Zn \\ |\v|_\infty \leq N}}|X(\v)|,
\end{equation}
where $|\v|_\infty = \max_{1\leq i \leq n}|v_i|$. We need to place a condition on the width of the stencil $F$.
\begin{definition}\label{def:Lipschitz}
We say $F:\H_n\to \R$ has \emph{width $N$} if there exists $C>0$ such that for all $X,Y\in \H_n$ we have
\[|F(X) - F(Y)| \leq C\|X - Y\|_{N,\infty}.\]
The smallest such constant $C>0$ is called the \emph{Lipschitz constant}  of $F$ and denoted $\Lip_N(F)$.
\end{definition}
The choice of $F$ given in \eqref{eq:ourF} is monotone and has width $N=1$, with $\Lip_1(F)=\frac{1}{2}$.

\subsection{Existence and uniqueness}\label{sec:exist}

We first establish a comparison principle.
\begin{theorem}\label{thm:discrete_comparison}
Assume $F$ is monotone and has width $N$. Let $u,v:\Zh \to \R$ satisfy
\begin{equation}\label{eq:subsuper}
u- F(\nabla^2_h u) \leq v - F(\nabla^2_h v) \ \ \text{on } \Zh,
\end{equation}
and
\begin{equation}\label{eq:growth_cond}
\lim_{|x|\to \infty} \frac{u(x) - v(x)}{|x|^2 }=0.
\end{equation}
Then $u\leq v$ on $\Zh$.
\end{theorem}
\begin{proof}
For $\epsilon>0$ we define
\[u_\epsilon(x) = u(x) - \frac{\epsilon}{2}|x|^2 - \epsilon N^2 - \epsilon.\] 
We claim that $u_\epsilon \leq v$ for all $\epsilon>0$, from which the result follows. Fix $\epsilon>0$ and assume by way of contradiction that $\sup_{\Zh}(u_\epsilon - v) > 0$. By \eqref{eq:growth_cond}, there exists $R>0$, depending on $\epsilon>0$, such that $u_\epsilon(x) < v(x)$ for $|x| >R$. Thus, $u_\epsilon - v$ attains its maximum over $\Zh$ at some $x_0\in \Zh$, and $u_\epsilon(x_0) > v(x_0)$. Since $u_\epsilon(x) - v(x) \leq u_\epsilon(x_0)  - v(x_0)$ for all $x\in \Zh$ we have
\[u_\epsilon(x) - u_\epsilon(x_0) \leq v(x) - v(x_0).\]
It follows that
\begin{align*}
\nabla^2_h u_\epsilon(x_0,\v) &= \frac{1}{h^2}(u_\epsilon(x_0+h\v) - u(x_0) + u_\epsilon(x_0-h\v) - u(x_0)) \\
&\leq \frac{1}{h^2}(v(x_0+h\v) - v(x_0) + v(x_0-h\v) - v(x_0)) = \nabla^2_h v(x_0,\v)
\end{align*}
for all directions $\v\in \Zn$. Since $F$ is monotone we have
\begin{equation}\label{eq:Fmon}
F(\nabla^2_h u_\epsilon(x_0)) \leq F(\nabla^2_h v(x_0)).
\end{equation}

We now compute
\[ \nabla^2_h u_\epsilon (x,\v) = \nabla^2_h u (x,\v) - \epsilon |\v|^2,\]
from which it follows that
\begin{align*}
u_\epsilon(x_0) - F(\nabla^2_h u_\epsilon(x_0)) &= u(x_0) - \frac{\epsilon}{2}|x_0|^2 -\epsilon N^2 - \epsilon - F(\nabla^2_h u(x_0) - \epsilon|\v|^2)\\
&\leq u(x_0) - F(\nabla^2_h u(x_0))  - \frac{\epsilon}{2}|x_0|^2- \epsilon\\
&<  v(x_0) - F(\nabla^2_h v(x_0))\\
&\leq v(x_0) -F(\nabla^2_h u_\epsilon(x_0)),
\end{align*}
where the last line follows from \eqref{eq:Fmon}.  Therefore $u_\epsilon (x_0) \leq v(x_0)$, which is a contradiction.
\end{proof}

Existence of a solution follows from the comparison principle and the Perron method.
\begin{theorem}\label{thm:existence}
Assume $F$ is monotone, has width $N$, and $F(0)=0$. Suppose there exists $C_g>0$ so that
\begin{equation}\label{eq:glinear}
|g(x)| \leq C_g (1 + |x|) \ \ \text{for all } x\in \Zh.
\end{equation}
Then there exists a unique solution $u:\Zh \to \R$ of \eqref{eq:general_discrete_pde} satisfying $\lim_{|x|\to \infty} \frac{u(x)}{|x|^2 }=0$. Furthermore, we have
\begin{equation}\label{eq:ulinear}
|u(x)| \leq C C_g(\Lip_N(F)N^3 + 1 + |x|) \ \ \text{for all } x\in \Zh,
\end{equation}
where $C$ depends only on $n$.
\end{theorem}
\begin{proof}
We define
\[\psi(x) =   \sqrt{1 + |x|^2},\]
and note that $\psi$ is a smooth function with linear growth satisfying
\begin{equation}\label{eq:psi_lower}
\frac{1}{ \sqrt{2}}(1+|x|) \leq \psi(x) \leq  1+|x| \ \ \text{for all } x\in \R^n.
\end{equation}
For any $\v\in \Zn$ and $x\in \Zh$ we have
\begin{align*}
|\nabla^2_h \psi(x,\v)| &\leq |\nabla^2_h \psi(x,\v) - \v^T \nabla^2 \psi(x)\v | + C|\v|^2\\
&\leq |\nabla^2_h \psi(x,\v) - \v^T \nabla^2 \psi(x)\v | + C|\v|^2\\
&\leq C\left(|\v|^3h+ |\v|^2\right) \leq C |\v|^3,
\end{align*}
where we used a Taylor series expansion of $\nabla^2_h \psi(x,\v)$ in the last line, along with $h\leq 1$ and $|\v|\geq 1$.  Since $F(0)=0$ we have
\begin{equation}\label{eq:psibound}
|F(\nabla^2_h\psi)|  = |F(\nabla^2_h\psi) - F(0)|\leq \Lip_N(F) \|\nabla^2_h \psi\|_{N,\infty}\leq C\Lip_N(F)N^3=:\xi.
\end{equation}
We now define
\[w = C_g(\xi +  \sqrt{2} \psi).\]
Then by \eqref{eq:psi_lower} and \eqref{eq:psibound}  we have
\[  w - F(\nabla^2_h w) = C_g\xi +  \sqrt{2}C_g \psi - F(C_g \nabla^2_h \psi) \geq  \sqrt{2}C_g \psi \geq C_g(1+|x|) \geq g.\]
A similar argument shows that $v=-w$ satisfies $v - F(\nabla^2_h v) \leq g$. 

Define
\begin{equation}\label{eq:perronset}
\F = \left\{v:\Zh \to \R \, : \, v - F(\nabla^2_h v) \leq g \ \text{and}  \ v \leq w\right\}
\end{equation}
and
\begin{equation}\label{eq:perronfunc}
u(x) = \sup\{v(x)\, : \, v \in \F\}.
\end{equation}
Since $v=-w$ belongs to $\F$, the set $\F$ is nonempty and we have $-w \leq u \leq w$. Hence, $u$ satisfies  \eqref{eq:ulinear}. 

We now claim that
\[u - F(\nabla^2_h u) \leq g \ \ \text{on } \Z_h.\]
To see this, fix $x_0 \in \Zh$ and let $v_k\in \F$ such that $\lim_{k\to \infty}v_k(x_0)=u(x_0)$. By passing to a subsequence, if necessary, we can assume that $\lim_{k\to \infty}v_k(x)$ exists for all $x\in \Zh$. Let us denote $v(x):=\lim_{k\to\infty}v_k(x)$, noting that $v(x_0)=u(x_0)$. By continuity we have $v - F(\nabla^2_h v) \leq g$ and $v \leq w$, thus $v\in \F$ and $v\leq u$. Since $v(x_0)=u(x_0)$, we have that $v-u$ attains its maximum at $x_0$. As in the proof of Theorem \ref{thm:discrete_comparison} we have $\nabla^2_hv(x_0,\v) \leq \nabla^2_h u(x_0,\v)$ for all $\v$, and so by the monotonicity of $F$ we have $F(\nabla^2_h v(x_0)) \leq F(\nabla^2_h u(x_0))$, which when combined with $u(x_0)=v(x_0)$ and $v - F(\nabla^2_h v)\leq g$, establishes the claim.

To complete the proof, we show that 
\[u - F(\nabla^2_h u) \geq g \ \ \text{on } \Zh.\]
Assume to the contrary that there is some $x_0\in \Zh$ such that
\[u(x_0) - F(\nabla^2_h u(x_0)) < g(x_0).\]
Define
\[v(x) = 
\begin{cases}
u(x_0) + \epsilon,& \text{if } x=x_0\\
u(x),& \text{otherwise.} 
\end{cases}\]
By continuity we can choose $\epsilon>0$ small enough so that 
\[v(x_0) - F(\nabla^2_h v(x_0)) \leq g(x_0).\]
For $x\neq x_0$, we already have
\[v(x) - F(\nabla^2_h v(x)) \leq g(x),\]
and the definition of $v$ and the monotonicity of $F$ imply that $\nabla^2 v(x,\v) \geq \nabla^2_h u(x,\v)$ for all $x\neq x_0$ and $\v\in \Zn$. This completes the proof.
\end{proof}

\subsection{Properties of solutions}\label{sec:properties}

We recall that a function $u:\Zh \to \R$ is \emph{Lipschitz continuous} if there exists $C>0$ such that
\[|u(x) - u(y)| \leq C\|x-y\|\]
holds for all $x,y  \in \Zh$. The \emph{Lipschitz constant} of $u$, denoted $\Lip(u)$, is the smallest such constant, given by
\[\Lip(u) = \sup_{\substack{x,y\in \Zh \\ x\neq y}} \frac{|u(x)-u(y)|}{\|x-y\|}.\]
\begin{lemma}[Basic properties]\label{lem:basicprop}
  Assume $F$ is monotone, has width $N$, and satisfies $F(0)=0$. Assume $g$ satisfies \eqref{eq:glinear} and let $u:\Zh\to \R$ be the unique solution of \eqref{eq:general_discrete_pde}. The following hold.
  \begin{enumerate}[\normalfont(i)]
    \item If $g$ is Lipschitz then so is $u$, and $\Lip(u) \leq \Lip(g)$.
    \item If $F\geq 0$ then $u\geq g$ on $\Zh$.
    \item There exists a constant $C>0$ such that
    \begin{equation}\label{eq:utog}
      \|u - g\|_\infty \leq C\Lip(g)\left( N\sqrt{\Lip_N(F)} + h\right).
    \end{equation}
  \end{enumerate}
\end{lemma}
\begin{proof}
  (i) Let $z\in \Zh$ and define $w(x) = u(x+z) - \Lip(g) \|z\|$. Then we have
  \[w(x) - F(\nabla^2_h w(x)) = u(x+z) - \Lip (g) \|z\| - F(\nabla^2_h u(x+z)) = g(x+z) - \Lip(g)  \leq g(x).\]
  By the comparison principle, Theorem \ref{thm:discrete_comparison}, we have $w \leq u$, and so
  \[u(x+z) - u(x) \leq \Lip(g)\|z\|.\]
  Since this holds for all $x,z\in \Zh$, the proof of (i) is complete.

  (ii) If $F\geq 0$ then $u(x) = g(x) + F(\nabla^2_h u(x)) \geq g(x)$.

  (iii) Let  $\bar{g}:\R^n\to \R$ be the piecewise constant extension of $g$ to a function on $\R^n$, defined so that $\bar{g}(y)=g(x)$ for all $y\in x+[0,h)^n$ and any $x\in Zh$. While the extension is discontinuous, it satisfies
  \[|\nabla_h \bar{g}(y,\v)| = \left| \frac{g(x+h\v) - g(x)}{h}\right| \leq \Lip(g)|\v|\]
  provided $\v\in \Zn$, where $x\in \Zh$ satisfies $y\in x+[0,h)^n$.   Let $\eps>0$ and define the standard mollification $g_\eps:=\eta_\eps * \bar{g}$, where $\eta_\epsilon(x) = \frac{1}{\epsilon^d}\eta\left( \frac{x}{\epsilon}\right)$, and $\eta_\epsilon$ is compactly supported in $B(0,\epsilon)$.  Note that
  \[\nabla_h g_\epsilon(x,\v) = \int_{\R^n} \eta_\eps(y) \nabla_h \bar{g}(x-y,\v) \, dy = \int_{\R^n }\eta_\eps(x-z) \nabla_h \bar{g}(z,\v) \, dz,\]
  where we used the change of variables $z=x-y$ above. We also have
  \[\nabla_h g_\epsilon(x,-\v) = \int_{\R^n} \eta_\eps(y) \nabla_h \bar{g}(x-y,-\v) \, dy = -\int_{\R^n }\eta_\eps(x-z - h\v) \nabla_h \bar{g}(z,\v) \, dz,\]
where we now use the change of variables $z=x-y - h\v$.  Therefore
\begin{align*}
  \nabla^2_h g_\epsilon(x,\v) &= \frac{1}{h}(\nabla_h \bar{g}_\epsilon(x,\v) + \nabla_h \bar{g}_\epsilon(x,-v)) \\
  &= -\int_{\R^n}\nabla_h \eta_\epsilon(x-z,-\v) \nabla_h \bar{g}(z,\v)\, dz.
\end{align*}
Therefore
  \[    |\nabla^2_h g_\epsilon(x,\v)| \leq \int_{B(x,\epsilon)\cup B(x-h\v,\epsilon)}|\nabla_h \eta_\epsilon(x-z,-\v)| |\nabla_h \bar{g}(z,\v)|\, dz \leq \frac{C}{\epsilon}\Lip(g)|\v|^2.\]
  Since $F$ has width $N$, it follows that
  \begin{equation}\label{eq:hessbound}
    |F(\nabla^2_h g_\epsilon(x))| = |F(\nabla^2_h g_\epsilon(x)) - F(0)| \leq \Lip_N (F)\|\nabla^2_h g_\epsilon\|_{N,\epsilon} \leq \frac{C}{\epsilon}\Lip_N(F)\Lip(g)N^2.
\end{equation}
  We also have that $|\bar{g}(y) - g(x)| \leq |y-x| +  \sqrt{d}h$, and so 
\begin{align*}
  |g_\epsilon(x) - g(x)| &= \left|\int_{B(x,\epsilon)} \eta_\epsilon(x-y)( \bar{g}(y) - g(x))\, dy  \right|\\
  &\leq \int_{B(x,\epsilon)} \eta_\epsilon(x-y)|\bar{g}(y) - g(x)|\, dy\\
  &\leq \Lip(g)(\epsilon +  \sqrt{d}h).
\end{align*}
Combining this with \eqref{eq:hessbound} we have 
  \[|g_\eps(x) + F(\nabla ^2 g_\eps(x)) - g(x)| \leq  \Lip(g)(\epsilon +  \sqrt{d}h) +  \frac{C}{\epsilon}\Lip_N(F)\Lip(g)N^2.\]
  Choosing $\epsilon=N \sqrt{\Lip_N(F)}$ yields
  \[|g_\eps(x) + F(\nabla ^2 g_\eps(x)) - g(x)| \leq C\Lip(g)\left( N\sqrt{\Lip_N(F)} + h\right).\]
  By the comparison principle, Theorem \ref{thm:discrete_comparison}, we have
  \[g_\epsilon - C\Lip(g)\left( N\sqrt{\Lip_N(F)} + h\right) \leq u \leq g_\epsilon + C\Lip(g)\left( N\sqrt{\Lip_N(F)} + h\right),\]
  which completes the proof.
\end{proof}

To study further properties of $u$, we need some additional definitions.
\begin{definition}\label{def:Fconvex}
We say that $F:\H_n\to \R$ is \emph{convex}  if for all $X,Y\in \H_n$ and $\lambda\in [0,1]$ we have
\[F(\lambda X + (1-\lambda)Y) \leq \lambda F(X) + (1-\lambda)F(Y).\]
\end{definition}
\begin{definition}\label{def:uconvex}
We say that $u:\Zh\to \R$ is \emph{convex} if  $\nabla^2_h u(x) \geq 0$ for all $x\in \Zh.$
\end{definition}
We note that by \eqref{eq:discrete_taylor}, the convexity of $u$ is equivalent to the inequality
\[u(x) \leq \frac{1}{2}(u(x+h\v) + u(x-h\v))\]
holding for all $x\in \Zh$ and $\v \in \Zn$. We also note that the choice of $F$ given in \eqref{eq:ourF} is convex.

We also consider a permutation invariance property.
\begin{definition}\label{def:permutationinvariance}
  We say $u:\Zh \to \R$ is \emph{permutation invariant} if $u \circ \sigma = u$ for all permutations $\sigma$ on $\{1,\dots,n\}$.
\end{definition}
\begin{definition}\label{def:Fpermutationinvariance}
  We say $F:\H_n \to \R$ is \emph{permutation invariant} if $F(X\circ \sigma)=F(X)$ for all $X\in \H_n$ and all permutations $\sigma$ on $\{1,\dots,n\}$.
\end{definition}
We note that $F$ given in \eqref{eq:ourF} is permutation invariant. 

Finally, we also study translation properties.
\begin{definition}\label{def:translationprop}
  For $\v\in \Zn$, we say $u:\Zh\to \R$ satisfies the \emph{$\v$-translation property} if there exists a constant $c_\v$ such that
\begin{equation}\label{eq:translation}
  u(x+s\v) = u(x) + c_\v s
\end{equation}
for all $x\in \Zh$ and $s\in \R$ such that $s\v \in \Zh$.
\end{definition}

The next lemma shows that these properties for $u$ are inherited from $F$ and $g$.
\begin{lemma}\label{lem:otherprop}
Assume $F$ is monotone, has width $N$, and satisfies $F(0)=0$. Assume $g$ satisfies \eqref{eq:glinear} and let $u:\Zh\to \R$ denote the unique solution of \eqref{eq:general_discrete_pde}. The following hold:
\begin{enumerate}[\normalfont(i)]
\item If $F$ and $g$ are convex, then $u$ is convex.
\item If $F$ and $g$ are permutation invariant, then $u$ is permutation invariant.
\item If $g$ satisfies the $\v$-translation property with constant $c_\v$, then $u$ does as well.
\end{enumerate}
\end{lemma}
\begin{proof}
(i)
Let $\v\in \Zn$ and define
\[w(x) = \frac{1}{2}(u(x+h\v) + u(x-h\v)).\]
Then for $x\in \Zh$ we compute, using the convexity of $F$, that
\begin{align*}
  w(x) - F(\nabla^2_h w(x)) &= \frac{1}{2}u(x+h\v) + \frac{1}{2}u(x-h\v) - F\left( \frac{1}{2}\nabla^2_h u(x+h\v) + \frac{1}{2}\nabla^2_h u(x-h\v)\right)\\
  &\geq  \frac{1}{2}u(x+h\v) + \frac{1}{2}u(x-h\v) - \frac{1}{2}F(\nabla^2_h u(x+h\v)) - \frac{1}{2}F(\nabla^2_h u(x-h\v)) \\
  &= \frac{1}{2}g(x+h\v) + \frac{1}{2}g(x-h\v)\\
  &= \frac{h^2}{2}\nabla^2_h g(x,\v) + g(x)\\
  &\geq g(x),
\end{align*}
where the last line follows from the convexity of $g$. By Theorem \ref{thm:discrete_comparison} we have $w \geq u$, and so
  \[\nabla^2_h u(x,\v) = \frac{2}{h^2}(w(x) - u(x)) \geq 0\]
  for all $x\in \Zh$. Since $\v\in \Zn$ is arbitrary, we have $\nabla^2_h u(x)\geq 0$ for all $x\in \Zn$, hence $u$ is convex.

(ii)
  Let $\sigma$ be a permutation and set $w = u\circ \sigma$. Then we have
  \[w(x) - F(\nabla^2_h w(x)) = u(\sigma(x)) - F(\nabla^2_h (u\circ \sigma)(x)).\]
  We note that for any $\v\in \Zn$ we have  $\nabla^2_h (u\circ \sigma)(x,\v) = \nabla^2_h u(\sigma(x),\sigma(\v))$. Therefore $\nabla^2_h (u\circ \sigma)(x) = \nabla^2_h u(\sigma(x)) \circ \sigma$. Since $F$ is permutation invariant we have
  \[F(\nabla^2_h (u\circ \sigma)(x)) = F(\nabla^2_h u(\sigma(x))\circ \sigma) = F(\nabla^2_h u(\sigma(x))).\] 
  Therefore
  \[w(x) - F(\nabla^2_h w(x)) = u(\sigma(x)) - F(\nabla^2_h u(\sigma(x))) = 0.\]
  By uniqueness of the solution $u$ of \eqref{eq:general_discrete_pde} we have $w=u$, which completes the proof.

(iii)
  Let $s\in \R$ such that $s\v \in \Zh$, and define $w(x) = u(x + s\v) - c_\v s$. Then 
  \[w(x) - F(\nabla^2_h w(x)) = u(x+s\v) - c_\v s - F(\nabla^2_h u(x+s\v)) = g(x+s\v) - c_\v s = g(x).\]
  Since the solution of \eqref{eq:general_discrete_pde} is unique, we have $w=u$, which completes the proof.
\end{proof}

\subsection{Convergence rates}\label{sec:convergence}

In this section, we prove convergence of the numerical scheme \eqref{eq:general_discrete_pde} towards the viscosity solution of the second order degenerate elliptic equation
\begin{equation}\label{eq:gpde}
  u - F(\nabla^2 u) = g \ \ \text{ on } \ \R^n.
\end{equation}
As before, we assume $F:\H_n\to \R$, and we interpret $F(X)$ for an $n\times n$ symmetric matrix $X$ as
\[F(X) := F(\v \mapsto \v^T X\v).\]
We recall the definitions of viscosity solutions in Appendix \ref{sec:defviscosity}. Throughout this section we will use the notation $\usc(\O)$ (resp.~$\lsc(\O)$) for the set of functions that are upper (resp.~lower) semicontinuous at all points in $\O \subset \R^n$.  For more details on the theory of viscosity solutions, we refer the reader to the user's guide \cite{crandall1992user} and \cite{calderviscosity}.

Existence and uniqueness of a linear growth viscosity solution to \eqref{eq:gpde} is standard material for viscosity solutions, and proofs can be found in \cite{crandall1992user,calderviscosity}. For use later on, we recall the comparison principle for \eqref{eq:gpde} in Lemma \ref{lem:ubcomp} below. The proof of this result is standard in viscosity solution theory; a self-contained proof can be found in \cite[Lemma 12.17]{calderviscosity}. 
\begin{lemma}\label{lem:ubcomp}
Assume $F$ is uniformly continuous, monotone, has width $N$, and satisfies $F(0)=0$. Let $u\in \usc(\R^n)$ be a viscosity subsolution of \eqref{eq:gpde} and let $v \in \lsc(\R^n)$ be a viscosity supersolution of \eqref{eq:gpde}. If 
\begin{equation}\label{eq:growthcond}
\lim_{|x|\to \infty}\frac{u(x)-v(x)}{|x|^2}=0
\end{equation}
then $u\leq v$ on $\R^n$.
\end{lemma}
Using this comparison principle and the Perron method, we can prove existence of a linear growth solution (Theorem \ref{thm:ub} below).  The application of the Perron method is standard and a self-contained proof can be found in  \cite[Theorem 12.18]{calderviscosity}.\footnote{In fact, the proofs do not require Lipschitzness; uniform continuity is sufficient.}
\begin{theorem}\label{thm:ub}
Assume $F$ is monotone, has width $N$, and satisfies $F(0)=0$. Assume $g$ is Lipschitz continuous and there exists $C_g>0$ such that
\begin{equation}\label{eq:glinear_growth}
  |g(x)| \leq C_g (1 + |x|).
\end{equation}
Then there exists a unique viscosity solution $u\in C(\R^n)$ of \eqref{eq:gpde} satisfying 
\[\lim_{|x|\to \infty}\frac{u(x)}{|x|^2} = 0.\]
Furthermore, there exists $C>0$ such that
\begin{equation}\label{eq:ugrowth}
|u(x)| \leq C(1 + |x|).
\end{equation}
\end{theorem}

Convergence of the discrete scheme \eqref{eq:general_discrete_pde} to the PDE \eqref{eq:gpde} is a standard result in viscosity solution theory. We state the theorem in the following result and briefly sketch the proof.
\begin{theorem}\label{thm:convergence}
Assume $F$ is monotone, has width $N$, and satisfies $F(0)=0$. Assume $g:\R^n\to \R$ is Lipschitz continuous and satisfies \eqref{eq:glinear_growth}. Let $u_h$ be the solution of \eqref{eq:general_discrete_pde} and let $u$ be the viscosity solution of \eqref{eq:gpde}. Then $u_h$ converges to $u$ locally uniformly as $h\to 0$, that is for all $R>0$ we have
\begin{equation}\label{eq:locallyuniform}
  \lim_{h\to 0} \max_{\substack{x\in \Zh \\ |x| \leq R}} |u_h(x) - u(x)| = 0.
\end{equation}
\end{theorem}
\begin{proof}
  By Lemma \ref{lem:basicprop} (i) we have $\Lip(u_h)\leq \Lip(g)$. By the Arzel\'a-Ascoli Theorem there exists a Lipschitz continuous function $u:\R^n\to \R$ such that, upon passing to a subsequence $u_{h_k}$, we have
  \[\lim_{k\to \infty} \max_{\substack{x\in \Z^n_{h_k} \\ |x| \leq R}} |u_{h_k}(x) - u(x)| = 0\]
  for all $R>0$. The proof will be completed by showing that $u$ is a viscosity solution of \eqref{eq:gpde}. By uniqueness of viscosity solutions of \eqref{eq:gpde}, the whole sequence $u_h$ converges locally uniformly to $u$.

We verify the subsolution property for $u$; the supersolution property is similar. Let $x_0 \in \R^n$ and $\phi\in C^\infty(\R^n)$ such that $u-\phi$ has a local maximum at $x_0$. Without loss of generality we can assume $x_0$ is a strict global maximum of $u-\phi$. It follows that there exists $x_k \to x_0$ such that $u_{h_k} - \phi$ attains its maximum over $\Z_{h_k}^n$ at $x_k$. Therefore $\nabla^2_h u_{h_k}(x_k,\v) \leq \nabla^2_h \phi(x_k)$, and since $F$ is monotone we have
  \[u(x_k) - F(\nabla^2_{h_k} \phi(x_k)) \leq \phi(x_k) - F(\nabla^2_h u_{h_k}(x_k)) =u(x_k) -u_{h_k}(x_k) - g(x_k).\]
  Sending $k\to \infty$ we obtain
  \[u(x_0) - F(\nabla^2 \phi(x_0)) \leq g(x_0),\]
  which completes the proof.
\end{proof}

If we have additional regularity for the solution $u$ of \eqref{eq:gpde}, we can prove convergence rates. We recall the $C^{k,1}$ seminorm of $u:\R^n\to \R$ is defined by
\[[u]_{C^{k,1}(\R^n)} = \sum_{1\leq |\alpha| \leq k} \Lip(D^\alpha u).\]
\begin{theorem}\label{thm:convergencerate}
  Assume $F$ is monotone, has width $N$, and satisfies $F(0)=0$. Assume $g:\R^n\to \R$ is Lipschitz continuous and satisfies \eqref{eq:glinear_growth}. Let $u_h$ be the solution of \eqref{eq:general_discrete_pde} and let $u$ be the viscosity solution of \eqref{eq:general_discrete_pde}. If $[u]_{C^{k,1}(\R^n)} < \infty$ for $k=2$ or $k=3$, then we have  
  \begin{equation}\label{eq:convergencerate}
    \|u - u_h\|_\infty \leq C\Lip_N(F)[u]_{C^{k,1}(\R^n)}N^{k+1}h^{k-1}
  \end{equation}
\end{theorem}
\begin{proof}
By Taylor expansion we have
  \[  |\nabla^2_h u(x,\v) - \v^T \nabla^2 u(x) \v| \leq C[u]_{C^{k,1}(\R^n)}|\v|^{k+1}h^{k-1}.\]
Therefore
  \[|F(\nabla^2_h u(x)) - F(\nabla^2 u(x))| \leq C\Lip_N(F)[u]_{C^{k,1}(\R^n)}N^{k+1}h^{k-1}=:\epsilon.\]
It follows that
  \[u(x) - F(\nabla^2_hu(x)) \leq u(x) - F(\nabla^2 u(x)) + \epsilon = g(x) + \epsilon \ \ \text{for } x\in \Zh.\]
  By Theorem \ref{thm:discrete_comparison} we have $u \leq u_h + \epsilon$ on $\Zh$. The opposite inequality is obtained similarly.
\end{proof}

From the convergence results in Theorems \ref{thm:convergence} and \ref{thm:convergencerate}, we immediately obtain that all the properties of the discrete solutions proved in Section \ref{sec:properties} extend to the viscosity solution $u$. 
\begin{proposition}\label{prop:properties}
Assume $F$ is monotone, has width $N$, and satisfies $F(0)=0$. Assume $g:\R^n\to \R$ is Lipschitz continuous and satisfies \eqref{eq:glinear_growth}. Let $u$ be the viscosity solution of \eqref{eq:general_discrete_pde}. Then the following hold. 
\begin{enumerate}[\normalfont(i)]
\item $u$ is Lipschitz continuous and $\Lip(u) \leq \Lip(g)$.
\item If $F\geq 0$ then $u\geq g$ on $\R^n$.
\item There exists a constant $C>0$ such that
\begin{equation}\label{eq:utog_continuum}
  \|u - g\|_\infty \leq CN\Lip(g)\sqrt{\Lip_N(F)}.
\end{equation}
\item If $F$ and $g$ are convex, then $u$ is convex.
\item If $F$ and $g$ are permutation invariant, then $u$ is permutation invariant.
\item If $g$ satisfies the $\v$-translation property with constant $c_\v$, then $u$ does as well.
\end{enumerate}
\end{proposition}
We note that in (iv), the notion of convexity for $F$ is the same as in Definition \ref{def:Fconvex}, while for $u$ it is the usual one for functions on $\R^n$, that is
\[u(\lambda x + (1-\lambda) y) \leq \lambda u(x) + (1-\lambda) u(y).\]
In (v), the definition of permutation invariance for $u:\R^n\to \R$ is identical to Definition \ref{def:permutationinvariance}; namely $u=u\circ \sigma$.  The translation property is defined similarly, but is slightly different so we give the definition below for functions on $\R^n$.
\begin{definition}\label{def:translationpropRn}
For $\v\in \R^n$, we say $u:\R^n\to \R$ satisfies the \emph{$\v$-translation property} if there exists a constant $c_\v$ such that \eqref{eq:translation} holds for all $x\in \R^n$ and $s\in \R$.
\end{definition}
The proof of Proposition \ref{prop:properties} follows from Lemmas \ref{lem:basicprop} and \ref{lem:otherprop}, and the convergence result in Theorem \ref{thm:convergence}.

\subsection{Restricting the domain}
\label{sec:restricting}

In order to compute the solution of the discrete scheme \eqref{eq:general_discrete_pde} on an unbounded domain $\Zh$, it is necessary to restrict the domain to a compact set. In this section, we study restrictions of \eqref{eq:pde} to \emph{computational} domains of the form $\Omega_{T,h} := [-T,T]^n\cap \Zh$. In this section we always assume $T=mh$ for some integer $m\geq 1$.  We define the width $N$ boundary as
\begin{equation}\label{eq:boundary}
  \partial_N \Omega_{T,h} = \{x\in \Zh\setminus \Omega_{T,h}\, : \, \text{there exists }y\in \Omega_{T,h} \text{ with }   |x-y|_\infty \leq N\}.
\end{equation}
We set a Dirichlet boundary condition on $\partial_N \Omega_{T,h}$ and show that this condition affects the solution only near the boundary, and the solution remains accurate in the interior of $\Omega_{T,h}$.

In this section we study the equation
\begin{equation}\label{eq:general_discrete_pde_T}
  u - F(\nabla ^2_h u) = g \ \  \text{ on }  \ \Omega_{T,h},
\end{equation}
which is the restriction of \eqref{eq:general_discrete_pde} to the computational domains $\Omega_{T,h}$. We first recall the comparison principle for \eqref{eq:general_discrete_pde_T}.
\begin{lemma}\label{lem:Tcomp}
Assume $F$ is monotone, has width $N$, and satisfies $F(0)=0$.  Let $u,v:\Zh\to \R$ satisfy
\begin{equation}\label{eq:subT}
  u - F(\nabla^2_h u) \leq v - F(\nabla^2_h v) \ \ \text{ on } \ \Omega_{T,h}
\end{equation}
and $u\leq v$ on $\partial_N \Omega_{T,h}$. Then $u\leq v$ on $\Omega_{T,h}$.
\end{lemma}
\begin{proof}
  Let $x_0\in \Omega_{T,h}\cup \partial_N \Omega_{T,h}$ be a point where $u-v$ attains its maximum value over $\Omega_{T,h}\cup \partial_N \Omega_{T,h}$. If $x_0 \in \partial_N \Omega_{T,h}$ then $u\leq v$, so we may assume $x_0 \in \Omega_{T,h}$. In this case we have $\nabla^2_h u(x_0,\v) \leq \nabla^2_h v(x_0,\v)$, and since $F$ is monotone, we obtain
  \[ u(x_0) - v(x_0) \leq F(\nabla^2_h u(x_0))  - F(\nabla^2_h v(x_0)) \leq 0.\]
 This completes the proof. 
\end{proof}

We now establish localization of solutions of \eqref{eq:general_discrete_pde_T}.
\begin{theorem}\label{thm:localization}
Assume $F$ is monotone, has width $N$, and satisfies $F(0)=0$.  Let $u,v:\Zh\to \R$ satisfy \eqref{eq:subT}. Then for any $\alpha \in (0,1)$ we have
\begin{equation}\label{eq:local}
  \max_{\Omega_{\alpha T,h}}(u-v) \leq \frac{C \Lip_N(F)N^2}{(1-\alpha)^2T^2}\left(\log\left( \frac{T^2}{\Lip_N(F) N^2}\right)^2+1 \right)\max_{\partial_N \Omega_{T,h}}(u-v).
\end{equation}
\end{theorem}
\begin{remark}\label{rem:interior-domain-accuracy}
The estimate \eqref{eq:local} in Theorem \ref{thm:localization} shows that the Dirichlet boundary conditions on $\partial_N \Omega_{T,h}$ have a limited domain of influence on the solution in $\Omega_{T,h}$ when $T>0$ is large. Indeed, if we fix, say, $\alpha=\tfrac12$, and if $u,v:\Zh\to \R$ are solutions of \eqref{eq:general_discrete_pde_T}, then $u$ and $v$ satisfy \eqref{eq:subT} with equality, and so it follows from two applications of Theorem \ref{thm:localization}, applied to $u-v$ and $v-u$, that
\begin{equation}\label{eq:localization}
  \max_{\Omega_{T/2,h}}|u-v|\leq \frac{C\log(T)^2}{T^2}\max_{\partial_N \Omega_{T,h}}|u-v|,
\end{equation}
holds, where $C$ depends on $\Lip_N(F)$ and $N$, and $T\geq 3$. This shows that errors in the Dirichlet condition on $\partial_N\Omega_{T,h}$ can be tolerated numerically, provided $T>0$ is large enough. For example,  by Lemma \ref{lem:basicprop} (iii), we can set $u=g$ on $\partial_N \Omega_{T,h}$ and obtain an $\O(\log(T)^2/T^2)$ approximation of the solution of \eqref{eq:gpde} on the interior domain $\Omega_{T/2,h}$. We show how to obtain even more accurate solutions with better choices of boundary conditions in Section \ref{sec:varying}.
\label{rem:localization} 
\end{remark}
\begin{proof}
  Let $\mu = \max_{\partial_N \Omega_{T,h}}(u-v)$ and define
\begin{equation}\label{eq:wsuper}
w(x) = v(x) + \gamma +  \mu \sum_{i=1}^n \left(e^{\lambda\left( \frac{x_i}{T}-1 \right)} + e^{-\lambda\left( \frac{x_i}{T}+1 \right)}\right)
\end{equation}
  for parameters  $\lambda,\gamma>0$ to be determined. For $x\in \partial_N \Omega_{T,h}$, there exists $i$ such that $x_i\geq T$ or $x_i\leq-T$, and so $w(x) \geq v(x) + \mu \geq u(x)$.   Note that we have
  \[|F(\nabla^2_h w) - F(\nabla^2_h v)| \leq \Lip_N(F) \|\nabla^2_h w - \nabla^2_h v\|_{N,\infty} \leq C\Lip_N(F)\frac{\mu\lambda^2N^2}{T^{2}}.\]
  Choosing  $\gamma = C\Lip_N(F)\mu\lambda^2N^2T^{-2}$ we find that
  \[w(x) - F(\nabla^2_h w(x)) \geq v(x) + \gamma - F(\nabla^2_h v(x)) - C\Lip_N(F)\mu\lambda^2N^2T^{-2} \geq 0.\]
  By Lemma \ref{lem:Tcomp} we have $w \geq u$ on $\Omega_{T,h}$, and so
\[u(x) - v(x) \leq \gamma + \mu \sum_{i=1}^n \left(e^{\lambda\left( \frac{x_i}{T}-1 \right)} + e^{-\lambda\left( \frac{x_i}{T}+1 \right)}\right)\]
for all $x\in \Omega_{T,h}$. For $x\in \Omega_{\alpha T,h}$ with $\alpha\in (0,1)$ we have
\[u(x) - v(x) \leq \gamma + 2n\mu e^{-\lambda(1-\alpha)}. \]
  Choosing $\lambda$ so that  $(1-\alpha)\lambda = \log\left( \frac{T^2}{\Lip_N(F) N^2}\right)$ completes the proof.
\end{proof}

\section{Numerical analysis of the prediction PDE}\label{sec:predictionanalysis}

This section is concerned with numerical analysis specific to the prediction from expert advice numerical scheme
\begin{equation}\label{eq:discrete_pde}
u(x) - \frac{1}{2}\max_{\v \in \B^n}\nabla^2_h u(x,\v) = g(x) \ \ \text{for } x\in \Zh.
\end{equation}
In particular, we show in Section \ref{sec:varying} how to use the translation property \eqref{eq:translation} to reduce the dimension by one, and in Section \ref{sec:sector} we show how to reduce the computational domain to a sector where the coordinates are ordered.

\subsection{Reducing the dimension}\label{sec:varying}

We show here how to reduce the problem from $n$ dimensional to $n-1$ dimensional. This requires the translation property \eqref{eq:translation} and is based on the following lemma.
\begin{lemma}\label{lem:complement}
Suppose $u$ satisfies the translation property \eqref{eq:translation} with $\v = \one$ and $c_\v = 1$.  Then for all $\v \in \B^n$ we have 
\begin{equation}\label{eq:complement}
  \nabla^2_h u(x,\v) = \nabla^2_h u(x,\one-\v).
\end{equation}
\end{lemma}
\begin{proof}
By assumption we have
\begin{equation}\label{eq:translationones}
u(x + s\one) = u(x) + s 
\end{equation}
for all $s\in \Z_h$ and $x\in \Zh$. We now compute, using \eqref{eq:translationones}, that
\begin{align*}
  h^2\nabla^2_h u(x,\v) &= u(x+h\v) - 2u(x) + u(x-h\v) \\
  &=u(x + h\v - h\one) + h - 2u(x) + u(x -h\v+h\one) - h\\
  &=u(x - h(\one-\v)) - 2u(x) + u(x + h(\one - \v))\\
  &=h^2 \nabla^2_h u(x,\one-\v),
\end{align*}
which completes the proof.
\end{proof}

For $x\in \R^{n-1}$ and $a\in \R$ we define $(x,a)\in \R^n$ by
\[(x,a) = (x_1,x_2,\dots,x_{n-1},a).\]
The next lemma shows how we can use the translation property \eqref{eq:translation} to reduce \eqref{eq:discrete_pde} to a similar equation in one less variable.
\begin{lemma}\label{lem:reduction}
Let $u$ be the solution of \eqref{eq:discrete_pde} and suppose $u$ satisfies the translation property \eqref{eq:translation} with $\v = \one$ and $c_\v = 1$. Define $w,f:\Zmh \to \R$ by  $w(x):=u(x,0)$ and $f(x)=g(x,0)$. The following hold.
\begin{enumerate}[\normalfont(i)]
\item The function $w$ satisfies
\begin{equation}\label{eq:reduction}
  w(x) - \frac{1}{2}\max_{\v \in \B^{n-1}}\nabla^2_h w(x,\v) = f(x) \ \ \text{for all } \ x\in \Zmh.
\end{equation}
\item   If $g$ is permutation invariant, then so is $w$, and furthermore $w$ satisfies
\begin{equation}\label{eq:translation_red}
  w(x)  =w\left(x - x_i(\one+\e_i)\right) + x_i \ \ \text{for all } x\in \Zmh \text{ and } i=1,\dots,n-1.
\end{equation}
\end{enumerate}
\end{lemma}
\begin{proof}
By Lemma \ref{lem:complement}, for any $x\in \Zmh$ we have
\[\max_{\v \in \B^n}\nabla^2_h u((x,0),\v)  = \max_{\v \in \B^{n-1}}\nabla^2_h u((x,0),(\v,0)) = \max_{\v \in \B^{n-1}}\nabla^2_h w(x,\v).\]
Since $u$ satisfies \eqref{eq:discrete_pde} we have
  \[w(x) - \frac{1}{2}\max_{\v \in \B^{n-1}}\nabla^2_h w(x,\v) = g(x,0)=f(x),\]
which completes the proof.

  (ii) Since $g$ is permutation invariant, so is $u$ (by Lemma \ref{lem:otherprop}) and hence $w$. Let $\sigma_{ij}$ denote the permutation swapping coordinates $i$ and $j$. Let $x\in \R^{n-1}$ and let $y=(x,0)$. For any $1 \leq i\leq n-1$ we use the translation property and permutation invariance to obtain
\begin{align*}
  w(x)=u( \left(x,0\right)) &= u\left( (x,0) - x_i\one\right) + x_i\\
  &= u\left( \left(x - x_i\one,-x_i\right)\right) + x_i\\
  &=u\left(\sigma_{i,n}\left(x - x_i\one,-x_i\right)\right) + x_i\\
  &=u\left(x - x_i\one - x_i\e_i,0\right) + x_i\\
  &=w\left(x - x_i\one - x_i\e_i\right) + x_i,
\end{align*}
which completes the proof.
\end{proof}

By Lemma \ref{lem:reduction} and Remark \ref{rem:interior-domain-accuracy} we may instead solve the equation
\begin{equation}\label{eq:w_pde}
\left\{
\begin{aligned}
w(x) - \frac{1}{2}\max_{\v\in \B^{n-1}}\nabla^2_h w(x,\v) &= g(x,0),&& \text{if }  x \in \Omega_{T,h}\\
w(x) &= g(x,0) ,&& \text{if } x \in \partial_1 \Omega_{T,h}
\end{aligned}
\right.
\end{equation}
in dimension $n-1$, where we take $T=mh$ to be a multiple of the grid resolution. Provided we restrict our attention to the localized interior set $\Omega_{\alpha T,h}$ for some $\alpha \in (0,1)$, then as per Remark \ref{rem:interior-domain-accuracy} we incur an $O(1/T^2)$ error term, up to logarithmic factors.

\subsection{Reducing the domain to a sector}\label{sec:sector}

It turns out that in addition to reducing the dimension of the equation from $n$ to $n-1$, we can also drastically reduce the size of the computational grid by restricting our attention to the sector  
\begin{equation}\label{eq:sector}
\D_n = \{x\in \Zh\, : \, x_1 \geq x_2  \geq \cdots \geq x_n\},
\end{equation}
and the positive  sector
\begin{equation}\label{eq:sector_positive}
\D^+_n = \{x\in \D_n\, : \, x_n \geq 0\}.
\end{equation}
Whenever $g$ is permutation invariant, e.g., the max-regret $g(x)=\max\{x_1,\dots,x_n\}$, the solution $u$ of \eqref{eq:discrete_pde} and the reduced solution $w(x)=u(x,0)$ are also permutation invariant. As we show in this section, this, combined we the translation property, or \eqref{eq:translation_red}, allows us to reduce the domain of the discrete PDE \eqref{eq:discrete_pde} to $\D_n^+$, which is drastically smaller than the full computational grid $\Omega_{T,h}$. 

In order to do this in a computational setting, for $x\in \D_{n-1}^+$ and $v\in \B^{n-1}$, we need to be able to evaluate $w(x+\v)$ and $w(x-\v)$ in terms of only the values of $w$ within the positive sector $\D_{n-1}^+$. This will allow us to evaluate the discrete second derivative $\nabla^2_h w(x,\v)$ without reference to the values of $u$ outside of $\D^{n-1}_+$. To do this, we need to define two operations. First, let $\pi_n:\R^n\to \R^n$ be the \emph{sorting function}  that sorts the coordinates of an $n$-dimensional vectors. Thus, $\pi_n(\Zh) = \D_n$. We also define the function $\xi:\D^n\to \R^n$ by 
\begin{equation}\label{eq:xi_def}
\xi_n(x) = 
\begin{cases}
x,& \text{if } x\in \D_n^+,\\
x +h(\one + \e_i),& \text{where } x_1,\dots,x_{i-1} \geq 0 > x_i\geq \cdots \geq x_n.
\end{cases}
\end{equation}
To indicate the cases above, we write
\[\one_{\xi_n}(x) = 
\begin{cases}
0,& \text{if } \xi_n(x) = x\\
1,& \text{otherwise.}
\end{cases}\]

The following lemma shows how to evaluate $w(x\pm \v)$ within the positive sector $\D_{n-1}^+$. 
\begin{lemma}\label{lem:boundary_cond_sector}
Suppose $w:\Z_h^{n-1} \to \R$ is permutation invariant and satisfies \eqref{eq:translation_red}. Let $T>0$,  $x\in \D_{n-1}^+ \cap [0,T-h]^{n-1}$ and $\v\in \B^{n-1}$. Then the following hold.
\begin{enumerate}[(i)]
\item For $y:=\pi_{n-1}(x+h\v)$ we have $y \in \D_{n-1}^+\cap [0,T]^{n-1}$ and $w(x + h\v) = w(y)$. 
\item For $y:=\pi_{n-1}(x-h\v)$ we have $\xi_{n-1}(y) \in \D_{n-1}^+\cap [0,T]^{n-1}$ and 
\[w(x - h\v) = w(\xi_{n-1}(y)) - h\one_{\xi_{n-1}}(y).\]
\end{enumerate}
\end{lemma}
\begin{proof}
Part (i) follows directly from the permutation invariance of $w$, and that $\v$ is a binary vector.

For part (ii), let $y = \pi_{n-1}(x - h\v) \in \D_{n-1}$. If $y\in \D_{n-1}^+$, then $\xi_{n-1}(y)=y$, $\one_{\xi_{n-1}}(y)=0$, and the result follows from the permutation invariance of $w$, as in part (i). If $y\not\in \D_{n-1}^+$, then $\xi_{n-1}(y)\neq y$ and $\one_{\xi_{n-1}}(y)=1$. Since $x\in \D_{n-1}^+$ we must have 
\[y_1\geq \cdots \geq y_{i-1}\geq 0 > -h=y_i = \cdots = y_{n-1},\]
for some $1 \leq i \leq n-1$. Therefore 
\[\xi_{n-1}(y) = y + h(\one + \e_i) = y - y_i(\one + \e_i),\]
and so by \eqref{eq:translation_red} and the permutation invariance of $w$ we have 
\[w(x - h\v) = w(y) = w(\xi_{n-1}(y)) + y_i = w(\xi_{n-1}(\pi_{n-1}(x-h\v))) - h,\]
which completes the proof.
\end{proof}

Lemma \ref{lem:boundary_cond_sector} allows us to restrict the reduced $n-1$ dimensional equation for $w$ to the sector $\D^+_{n-1}$, yielding the equation
\begin{equation}\label{eq:sector_pde}
\left\{
\begin{aligned}
w(x) - \frac{1}{2}\max_{\v\in \B^{n-1}}\nabla^2_h w(x,\v) &= g(x,0),&& \text{if }  x\in \D^+_{n-1}\cap \{x_1 \leq T - h\}\\
w(x) &= g(x,0) ,&& \text{if } x \in \D^+_{n-1}\cap \{x_1=T\},
\end{aligned}
\right.
\end{equation}
where $T=mh$ is a multiple of the grid resolution. By Lemma \ref{lem:boundary_cond_sector} we can compute the derivative $\nabla^2_h w(x,\v)$ for $x\in \D^+_{n-1}\cap \{x_1 \leq T - h\}$ via
\begin{equation}\label{eq:boundary_cond_d2v}
\nabla^2_h w(x,\v) = \frac{w(y_+) + w(\xi_{n-1}(y_-)) - h\one_{\xi_{n-1}}(y_-) - 2w(x)}{h^2},
\end{equation}
where $y_+ = \pi_{n-1}(x + h\v)$ and $y_- = \pi_{n-1}(x - h\v)$. By Lemma \ref{lem:boundary_cond_sector} the expression on the right hand side of \eqref{eq:boundary_cond_d2v} involves evaluating $w$ only at points in the sector $\D^+_{n-1}\cap\{x_1\leq T\}$. This is essentially equivalent to setting boundary conditions on the sector $\D^+_{n-1}$, though the explicit identification of those boundary conditions is a complicated task that we do not undertake here.

Now, the number of grid points in the full computational grid $[-T,T]^n\cap \Z_h^n$ grows exponentially in the dimension $n$ as $\O((2Th^{-1})^n)$, which is known as the \emph{curse of dimensionality}. However, the number of grid points in the sector $\D^+_{n-1}\cap\{x_1\leq T\}$ grows much slower in $n$, as the following lemma shows. 
\begin{lemma}\label{lem:num_grid_points}
Let $G_{n,T}$ denote the number of grid points in the sector $\D^+_{n}\cap\{x_1\leq T\}$, where $T=mh$ with $m$ a positive integer. Then $G_{n,T}\leq \frac{1}{n!}(Th^{-1} + n)^n$.
\end{lemma}
\begin{proof}
Let $T = mh$ and define
\begin{equation}\label{eq:amn}
a_{m,n} = \sum_{i_1=0}^{m}\sum_{i_2=0}^{i_1}\cdots \sum_{i_{n-1}=0}^{i_{n-2}}\sum_{i_n=0}^{i_{n-1}}1.
\end{equation}
Note that $G_{n,T}=a_{m,n}$. Hence, the proof boils down to showing that $a_{m,n}\leq \frac{1}{n!}(m+n)^{n}$. We will prove this with induction on $n$, using the recursive identity
\[a_{m,n}=\sum_{i=0}^m a_{i,n-1},\]
which follows directly from \eqref{eq:amn}. For the base step of $n=1$, we have $a_{m,1} = m+1$ by direct computation.  For the inductive step, assume that for some $n\geq 1$ and all $m$ we have $a_{m,n}\leq \frac{1}{n!}(m+n)^{n}$. Then we compute
\[a_{m,n+1} = \sum_{i=0}^m a_{i,n} \leq \sum_{i=0}^m \frac{1}{n!}(i+n)^n \leq \int_0^{m+1}\frac{1}{n!}(x + n)^n \, dx = \frac{1}{(n+1)!}(m+n+1)^{n+1}, \]
which completes the inductive step and hence the proof.
\end{proof}

If we choose $h$ small enough so that $n \leq Th^{-1}$, then the bound in Lemma \ref{lem:num_grid_points} implies 
\[G_{n,T} \leq \frac{(2 T h^{-1})^n}{n!}.\]
This is a factor of $n!$ smaller than the number of grid points on the full computational domain $[-T,T]^n\cap \Z_h^n$, which reflects the exponentially small size of the sector $\D^+_{n}$.  In order to understand better how $G_{n,T}$ scales with $n$ in Lemma \ref{lem:num_grid_points}, we use a version of Stirling's formula $n! \geq \sqrt{2\pi n} (n/e)^n$ \cite{robbins1955remark} to obtain
\[G_{n,T}\leq \frac{1}{n!}(Th^{-1} + n)^n \leq \frac{e^n}{\sqrt{2\pi n}}(1 + Th^{-1}n^{-1})^n \leq \frac{e^{n + Th^{-1}}}{\sqrt{2\pi n}}.\]
While this complexity is still exponential in $n$ and $h^{-1}$, these two quantities do not directly interact. For example, with the full grid $[-T,T]^n\cap \Z_h^n$, increasing from $n$ to $n+1$ dimensions requires a factor of $\O(h^{-1})$ more grid points, while for the sector $\D^+_{n}\cap\{x_1\leq T\}$ we require at most $e$ times as many grid points, which is independent of the grid resolution $h$. However, it is important to point out that the curse of dimensionality is not overcome. It is rather more accurate to say that we have postponed the curse to larger values of $n$. For example, in Section \ref{sec:numerics}, we conduct experiments with up to $n=7$ experts, using the reduced computational grid in dimension $n-1=6$, with a grid resolution of $h=0.1$. Using the full computational grid we are only able to compute solutions of the $n=4$ expert problem, on the reduced $n-1=3$ dimensional grid. 

\begin{remark}\label{rem:}
It is important to point out that the computational costs are more expensive than the memory required to store the grid, since we must compute the maximum over $\v\in \B^n$ in the operator, which is a maximum over $2^n$ directions. Thus, the computational cost admits an additional $2^n$ factor over the memory costs. 
\end{remark}

\section{Numerical results}\label{sec:numerics}

In this section we present numerical results for $n=2$ up to $n=10$ experts. In all cases, we solve the reduced equation \eqref{eq:reduction} for $w(x)=u(x,0)$, which is a problem in $n-1$ dimensions. In Section \ref{sec:full_grid} we present experiments on the full computational grid, while in Section \ref{sec:sector_grid} we present results on the smaller and more efficient sector grid. In each case, to solve the equation $w - F(\nabla^2_h w)  = g$, we iterate the fixed point scheme
\begin{equation}\label{eq:wk_iter}
w_{k+1} = (1-dt)w_k + dt(F(\nabla^2_h w_k) + g)
\end{equation}
until 
\[|w - F(\nabla^2_h w)  - g| \leq \frac{h^2}{100}.\]
In this section we use $F$ defined in \eqref{eq:ourF}. In all cases, we use $T=5$ and restrict the solutions to the unit box $[0,1]^{n-1}$. The code for all experiments is available on GitHub: \url{https://github.com/jwcalder/PredictionPDE}.

\subsection{Full computational grid}\label{sec:full_grid}

\begin{figure}[!t]
\centering
\subfloat[$n=2$ experts]{
\includegraphics[width=0.39\textwidth]{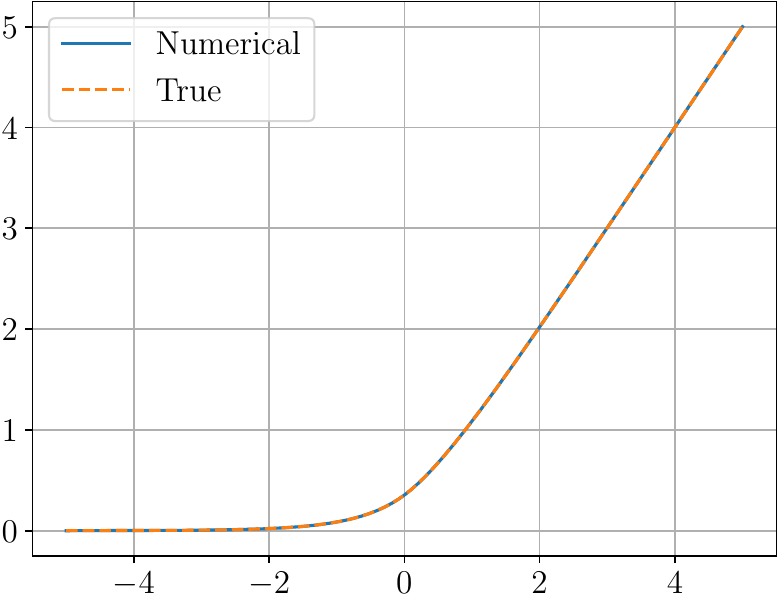}
\includegraphics[width=0.45\textwidth]{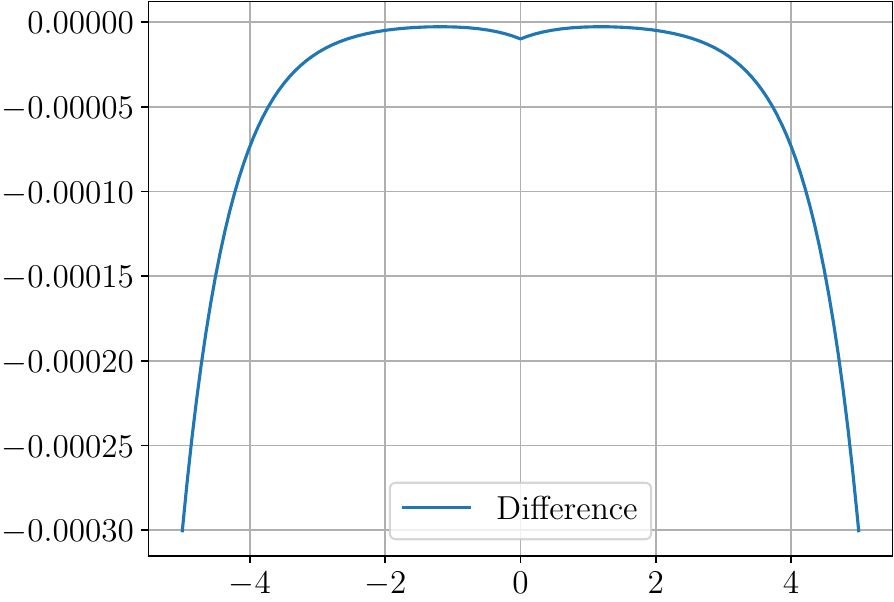}}\\
\subfloat[$n=3$ experts (numerical solution, true solution, difference) ]{
\includegraphics[width=0.32\textwidth]{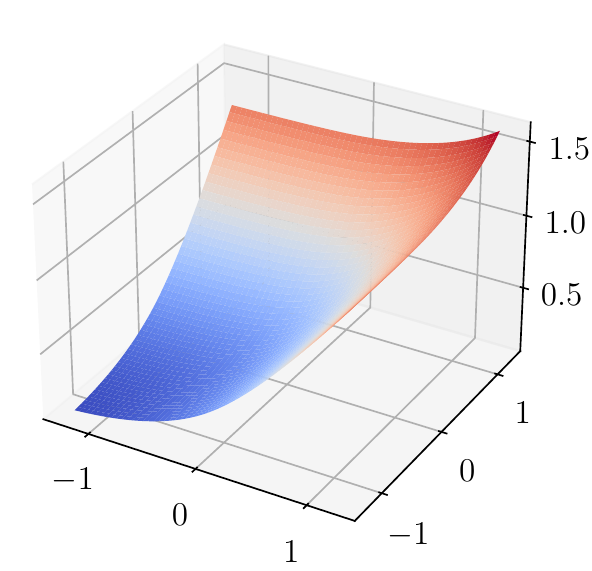}
\includegraphics[width=0.32\textwidth]{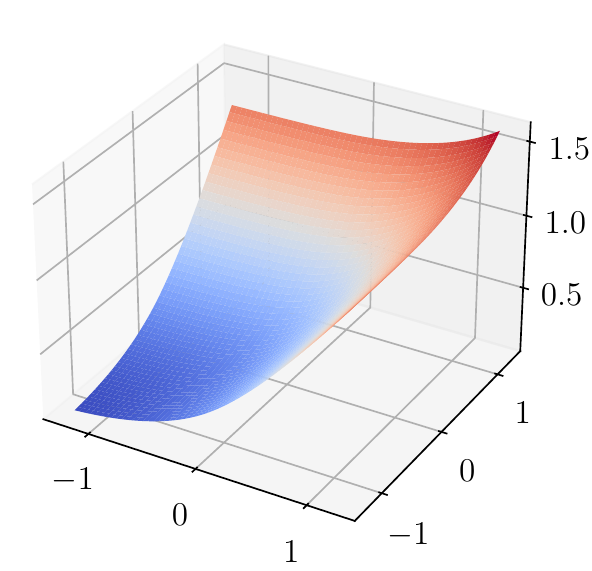}
\includegraphics[width=0.32\textwidth]{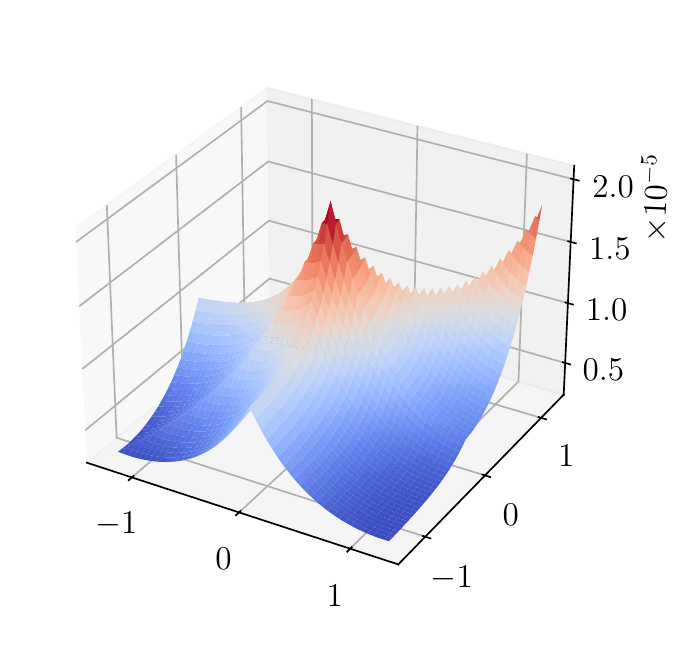}}
\caption{Plots of the numerical solution $w$ versus the true solutions for $n=2$ and $n=3$ experts.}
\label{fig:demo}
\end{figure}

We first present results on the full computational grid, so we solve the equation \eqref{eq:w_pde}. Due to the curse of dimensionality, we can only solve the equation for $n=2,3,4$ experts. The finest grid we used was $h=0.01$ for $n\leq 3$ and $h=0.025$ for $n=4$. In Figure \ref{fig:demo} we show plots of the numerical solutions versus the true solutions for $n=2,3$ experts. Since we solve the reduced $d=n-1$ dimensional equation, these are PDEs in $d=1$ and $d=2$ dimensions. The $n=2$ expert solution is accurate on the full domain $[-5,5]$ since the boundary condition $u(\pm 5) = \max\{5,0\}$ is exponentially accurate; see \eqref{eq:solution2experts}. For $n=3$ experts, the solution loses accuracy away from the restricted domain $[-1,1]^2$. 

\begin{figure}[!t]
\centering
\subfloat[Convergence rates]{\includegraphics[width=0.48\textwidth]{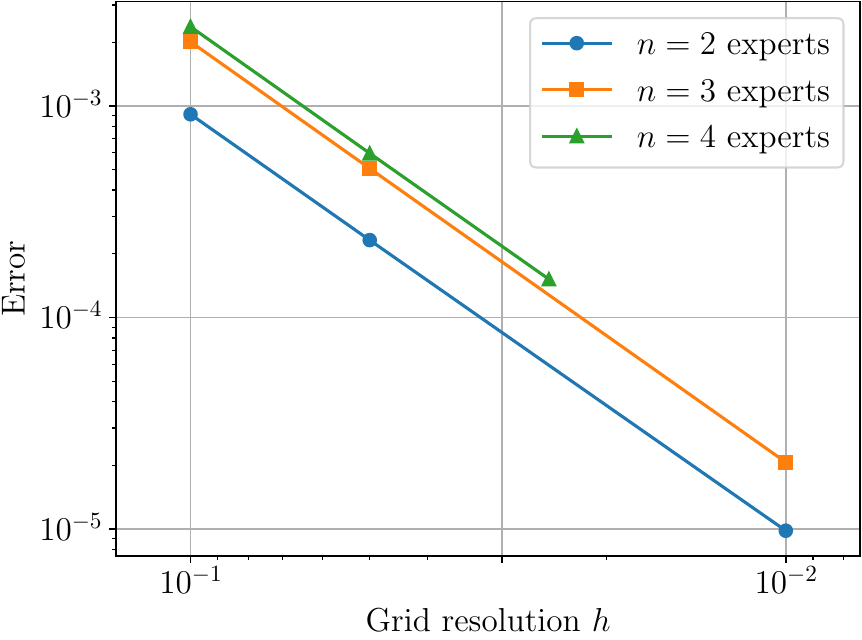}}
\hfill
\subfloat[$n=2$ expert strategies]{\includegraphics[width=0.465\textwidth]{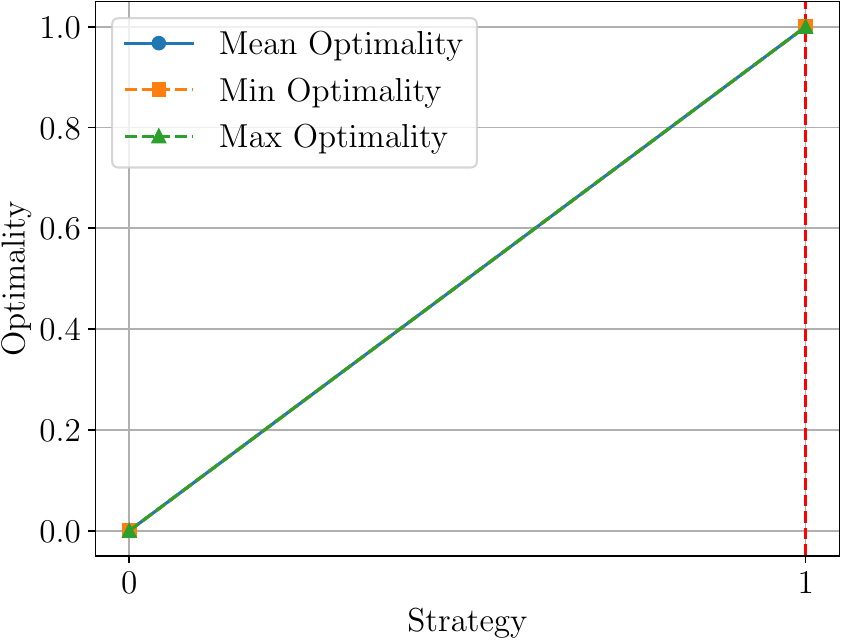}}\\
\subfloat[$n=3$ expert strategies]{\includegraphics[width=0.465\textwidth]{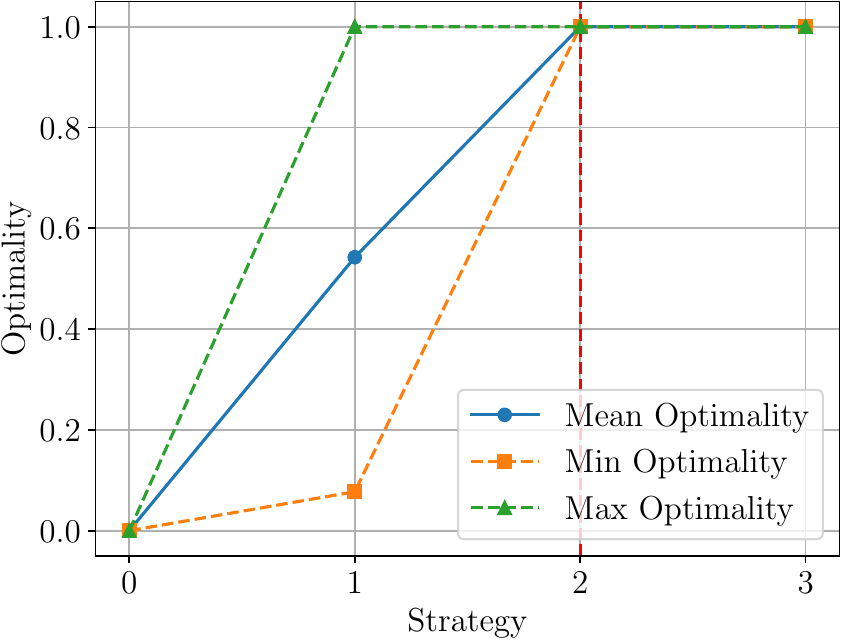}}
\hfill
\subfloat[$n=4$ expert strategies]{\includegraphics[width=0.465\textwidth]{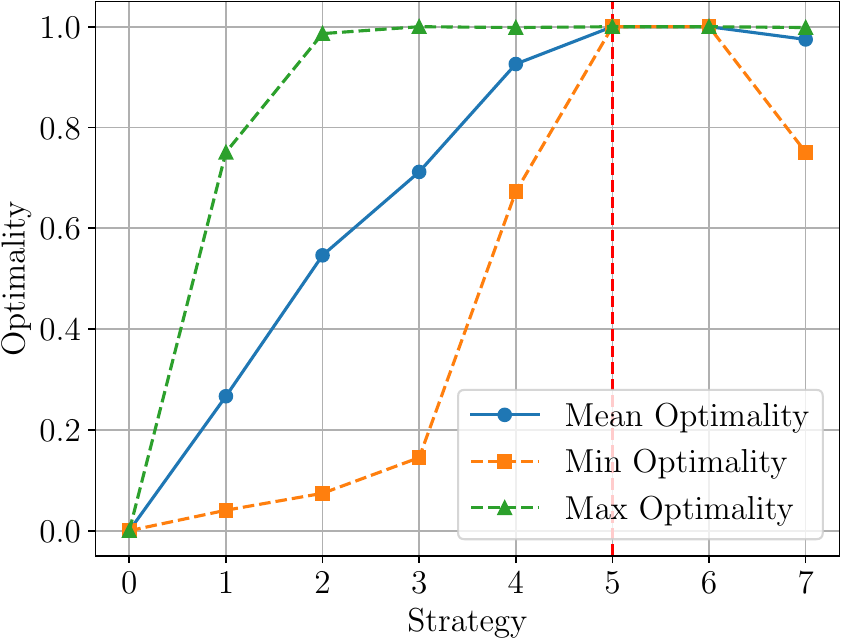}}
\caption{Convergence rates and optimal strategies for $n=2,3,4$ experts, computed from the numerical solutions. The dashed red line indicates the COMB strategy, which is numerically observed to be optimal for $n=2,3,4$ experts, as the theory predicts. }
\label{fig:opt_strat}
\end{figure}

In Figure \ref{fig:opt_strat} (a) we show a convergence analysis for varying grid resolution $h$ for $n=2,3,4$ experts, where the exact solutions of the PDEs are known. In all cases we observe second order $\O(h^2)$ convergence rates. Figure \ref{fig:opt_strat} also shows the optimality of each adversarial strategy. We define the optimality of strategy $\v\in \B^{n-1}$ at a grid point $x$ by the ratio
\[\opt(x,\v) = \frac{\nabla^2_h w(x,\v)}{\max_{\p\in \B^{n-1}}\nabla^2_h w(x,\p)} = \frac{\nabla^2_h w(x,\v)}{2(u(x) - g(x,0))}.\]
An optimality value of $\opt(x,\v)=1$ indicates that strategy $\v$ is optimal at grid point $x$. In order to measure the optimality of all the competing strategies, we plot the minimum, maximum, and average optimality values over the intersection of the computational grid with the positive sector $\D^{+}_{n-1}$. Since the solution is permutation invariant, the scores are the same in all other sectors.  If the \emph{minimum} score is $1$, up to the numerical precision $O(h^2)$, then that strategy is globally optimal over the unit box $[-1,1]^{n-1}$. In Figure \ref{fig:opt_strat} we denote the strategies, which are binary vectors, by the decimal number that strategy corresponds to, and we indicate the COMB strategy with a dashed red line. Since $\v$ and $\one-\v$ are equivalent strategies, we only show the optimality scores for the first half of the strategies; the plot for the second half is a mirror image of the plots shown.

In Figure \ref{fig:opt_strat} (b) we see that strategies $1=(0,1)$ and $2=(1,0)$ are optimal, both of which correspond to the COMB strategy. In Figure \ref{fig:opt_strat} (c) we see that for $n=3$ experts, the COMB strategy $2=(0,1,0)$ is optimal, as well as the non-COMB strategy $3=(0,1,1)$. In this case, we recall that the complement strategies $(1,0,1)$ and $(1,0,0)$ are equivalent, and hence also optimal, but are not depicted. For $n=4$ experts in Figure \ref{fig:opt_strat} (d) we see that the COMB strategy $5=(0,1,0,1)$ is optimal, as well as the non-COMB strategy $6=(0,1,1,0)$. All of these results have already been established theoretically in previous work; see Section \ref{sec:background}. We presented these results to verify that the numerical solvers are working properly and give results that agree with previous work.

We are unable to solve the PDE \eqref{eq:w_pde} for $w$ on a full computational grid for the $n=5$ expert problem, so for this we must resort to the sparse grid method that restricts attention to the positive sector. 

\subsection{Sparse grids}\label{sec:sector_grid}

We now turn to computations on the positive sector $\D^+_{n-1}\cap [0,T]^{n-1}$, which is far smaller than the full grid and allows us to run experiments for $n=5,6,7,8,9,10$ experts. In this case, we are solving equation \eqref{eq:sector_pde} using the methods described in Section \ref{sec:sector}. To facilitate computations, we flatten the sector $\D^+_{n-1}\cap [0,T]^{n-1}$ to a one-dimensional array, and store the solution $w$ as a one dimensional array with linear indexing. We pre-computed and stored the stencils for second derivatives in all directions $\v\in \B^{n-1}$ using the methods outlined in Section \ref{sec:sector} prior to the running the iteration \eqref{eq:wk_iter} to solve the equation. All code is written in Python using the Numpy package and fully vectorized operations. 

\begin{table}
\centering
\begin{tabular}{||c|c|c|c||}
\hline
Dimension ($d=n-1$) & Grid resolution $h$ & Sector $\D^+_d\cap [0,5]^d$ & Full grid $\Z_h^d\cap[-5,5]^d$\\\hhline{|=|=|=|=|}
4&0.025&$7\times 10^7$&$3\times 10^{10}$\\\hline
5&0.050&$1\times 10^8$&$3\times 10^{11}$\\\hline
6&0.100&$3\times 10^7$&$1\times 10^{12}$\\\hline
7&0.200&$3\times 10^6$&$8\times 10^{11}$\\\hline
8&0.250&$3\times 10^6$&$7\times 10^{12}$\\\hline
9&0.350&$8\times 10^5$&$1\times 10^{13}$\\\hline
\end{tabular}
\caption{Number of grid points in the sector computational domain $\D^+_d\cap [0,5]^d$ compared to the full grid $\Z_h^d\cap[-5,5]^d$. We use fewer grid points as the dimension increases since evaluating the PDE involves computing derivatives in $2^d$ directions, so the computational time and memory storage increase exponentially with $d$.}
\label{tab:num_grid_points}
\end{table}

\begin{figure}[!t]
\centering
\subfloat[$n=5$ expert strategies]{\includegraphics[width=0.5\textwidth]{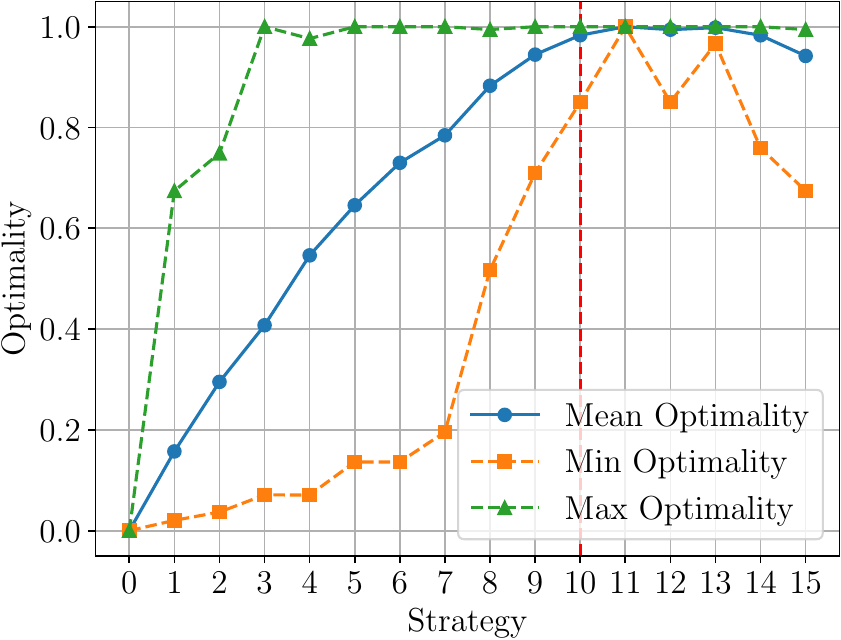}}\\
\subfloat[$n=6$ expert strategies]{\includegraphics[width=\textwidth]{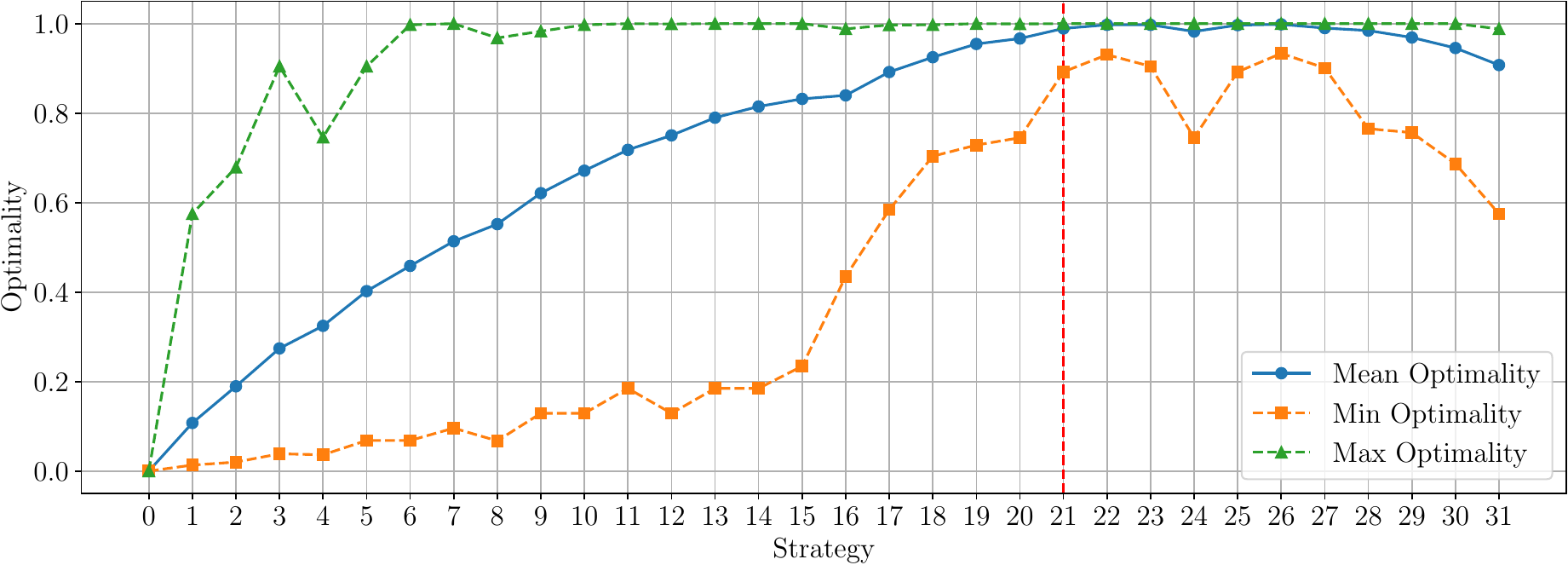}}
\caption{Numerical computation of strategy optimality for the $n=5,6$ expert problems.}
\label{fig:sparse_strat56}
\end{figure}

\begin{figure}[!t]
\centering
\subfloat[$n=7$ expert strategies]{\includegraphics[width=\textwidth]{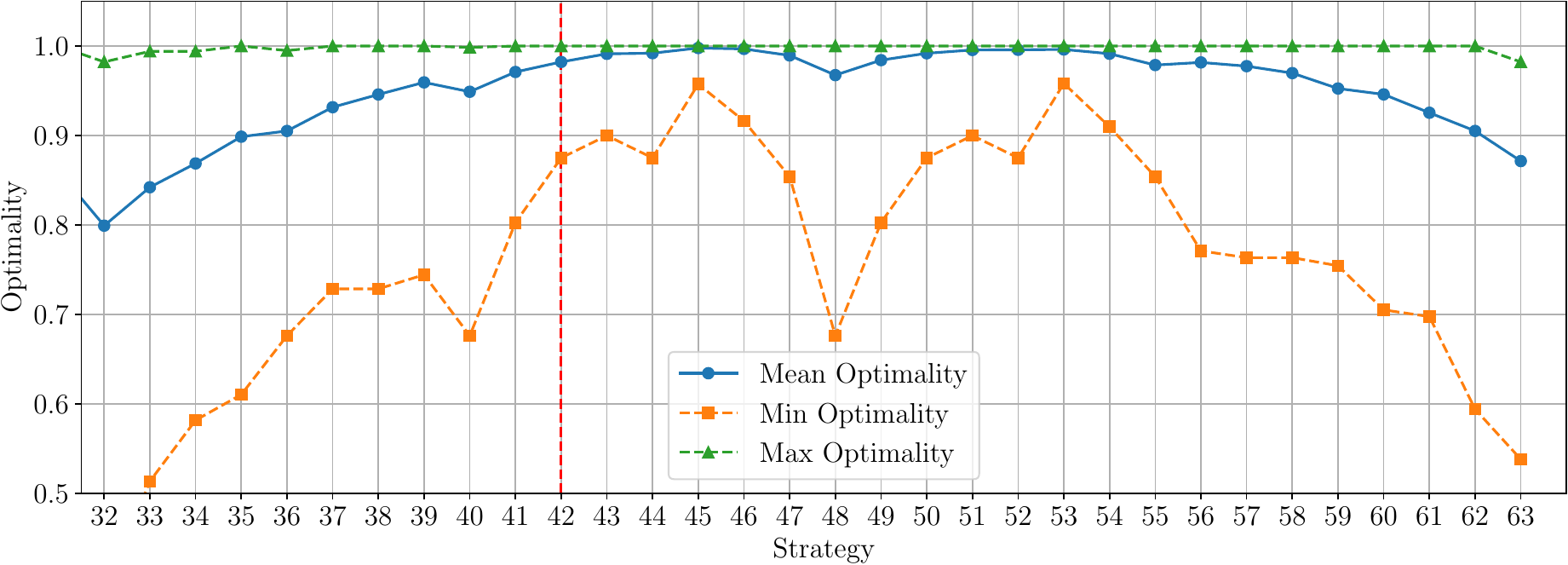}}\\
\subfloat[$n=8$ expert strategies]{\includegraphics[width=\textwidth]{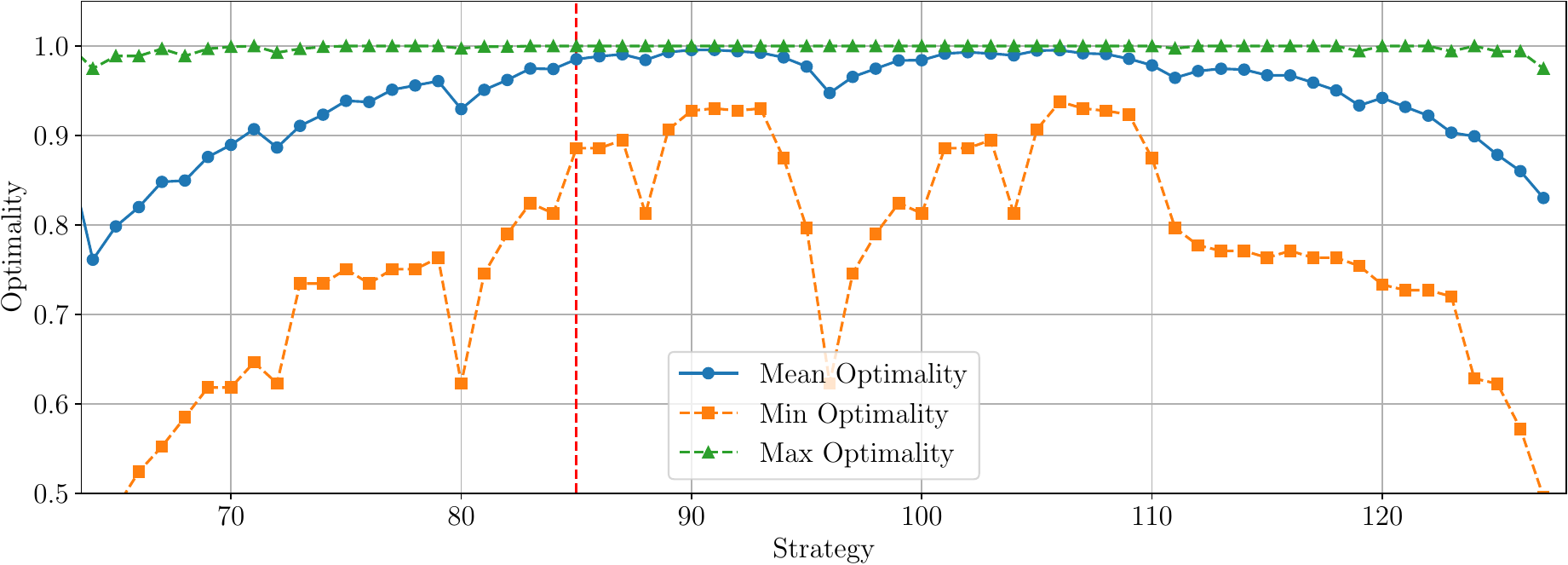}}
\caption{Numerical computation of strategy optimality for the $n=7,8$ expert problems. }
\label{fig:sparse_strat78}
\end{figure}

\begin{figure}[!t]
\centering
\subfloat[$n=9$ expert strategies]{\includegraphics[width=\textwidth]{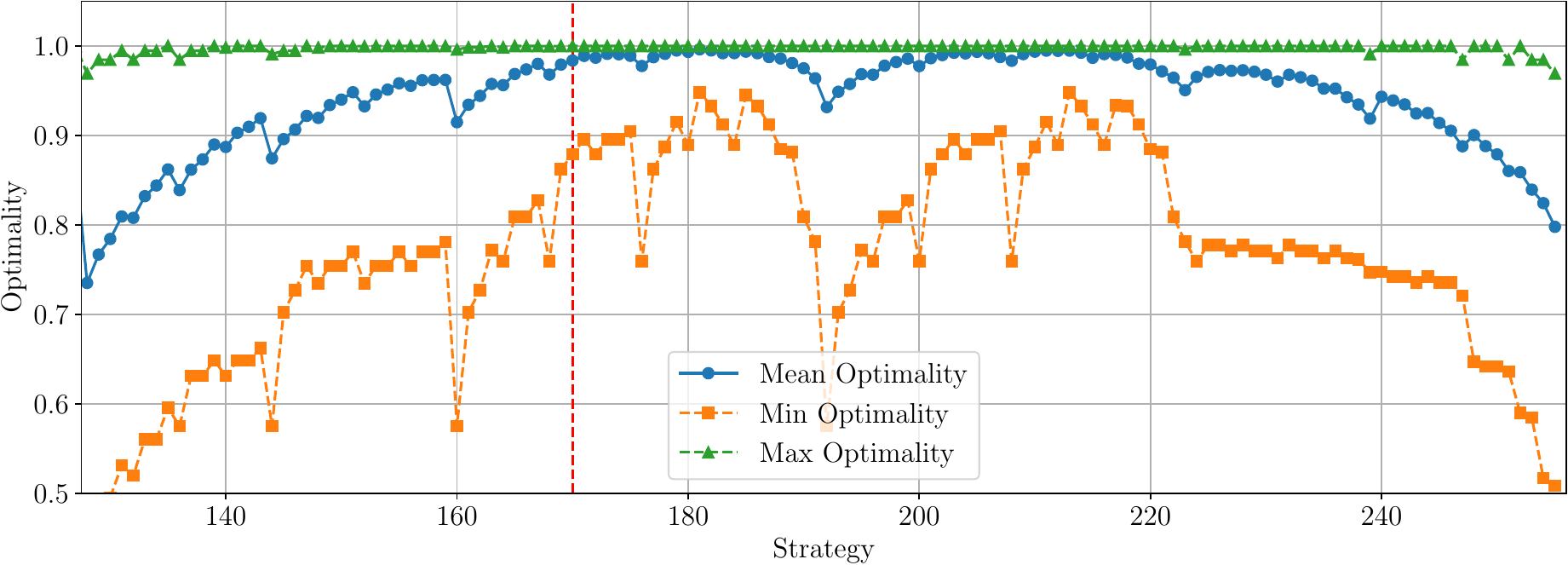}}\\
\subfloat[$n=10$ expert strategies]{\includegraphics[width=\textwidth]{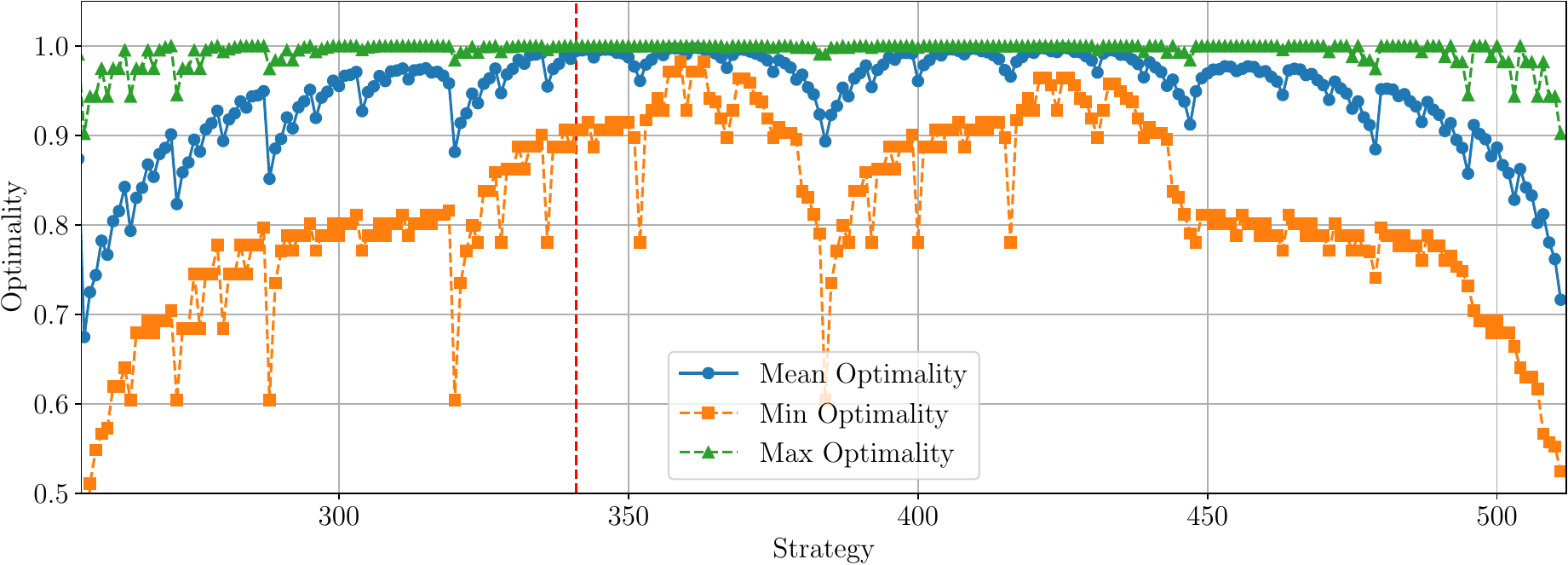}}\\
\caption{Numerical computation of strategy optimality for the $n=9,10$ expert problems. }
\label{fig:sparse_strat910}
\end{figure}

We used grid resolutions of $h=0.025$ for $n=5$, $h=0.05$ for $n=6$ and $h=0.1$ for $n=7$, $h=0.2$ for $n=8$, $h=0.25$ for $n=9$, and $h=0.35$ for $n=10$ experts. Table \ref{tab:num_grid_points} shows the number of grid points used in each dimension, compared to the number of grid points that would be required on the full grid. The simulations required between 25GB and 75GB of memory and each took less than one day to run on a single processor. Figures \ref{fig:sparse_strat56}, \ref{fig:sparse_strat78}, and \ref{fig:sparse_strat910} show the numerically computed optimality of the adversary's strategies for $n=5,6,7,8,9,10$ experts. In all cases, we have strong numerical evidence to indicate that the COMB strategy is \emph{not} globally optimal. This corroborates numerical evidence from \cite{chase2019experimental} for $n=5$.

We have strong numerical evidence in Figure \ref{fig:sparse_strat56} that the strategy $11=(0,1,0,1,1)$ is optimal for the $n=5$ expert problem over the box $[-1,1]^5$. Furthermore, the numerical evidence points to this being the \emph{only} optimal strategy, over the unit box, for $n=5$ experts. The minimum optimality score for strategy $11$ is $0.9999999999991459$, which is far more accurate than the second order $O(h^2)$ accuracy for $h=0.05$ would suggest. The second best competing strategy is $13=(0,1,1,0,1)$ with a minimum optimality score of $0.966$, which is well outside the range of numerical precision, suggesting that strategy $13$ is not globally optimal. 

The remaining plots in Figures \ref{fig:sparse_strat56}, \ref{fig:sparse_strat78}, and \ref{fig:sparse_strat910} provide very strong numerical evidence that there are no globally optimal adversarial strategies for $n=6,7,8,9,10$ experts. The highest minimum optimality scores are well outside of numerical precision. The one exception is the $n=10$ expert problem, where there are strategies with minimum optimality scores above $0.98$, while the grid resolution of $h=0.35$ yields $h^2 = 0.1225$. However, we expect this is an artifact from using a very coarse grid. We observed a similar phenomenon with $n=9$ experts, where some strategies appeared more optimal on a coarse grid.

\section{Conclusion and future work}

This paper developed and analyzed a numerical scheme to solve a degenerate elliptic PDE arising from prediction with expert advice in relatively high dimensions ($n \leq 10$) by exploiting symmetries in the equation and solution. Based on numerical results, we are able to make a number of conjectures for the optimality of various adversarial strategies in Section \ref{sec:main}. Our results have some limitations; mainly we are not able to solve the PDE on all of $\R^n$, and so our results are restricted to the box $[-1,1]^n$.

We expect these numerical methods could be extended to a few more experts, perhaps the $n=11$ and $n=12$ expert problems, using parallel processing or computational clusters with vastly more memory. The finite horizon problem should be amenable to similar techniques. There are also other prediction with expert advice PDEs, in particular the history-dependent experts setting \cite{calder2021asymp,drenska2021prediction}, that would benefit from numerical explorations. In terms of theory, we posed a number of conjectures and open problems that stem from this work in Section \ref{sec:main} that would be interesting to explore in future work. 

\appendix

\section{Definition of a viscosity solution}\label{sec:defviscosity}

For convenience, we recall the definitions of viscosity solutions for a general second order nonlinear partial differential equation
\begin{equation}\label{eq:hjbgen2}
H(\nabla ^2u,\nabla u,u,x) = 0 \ \ \text{ in } \O,
\end{equation}
where $H$ is continuous and $\O \subset \R^n$. Let $\usc(\O)$ (resp.~$\lsc(\O)$) denote the collection of functions that are upper (resp.~lower) semicontinuous at all points in $\O$. We make the following definitions.

We first recall the test function definition of viscosity solution.
\begin{definition}[Viscosity solution]\label{def:vis}
We say that $u \in \usc(\O)$ is a \emph{viscosity subsolution} of \eqref{eq:hjbgen2} if for every $x \in \O$ and every $\phi \in C^\infty(\R^n)$ such that $u - \phi$ has a local maximum at $x$ with respect to $\O$
\[H(\nabla ^2\phi(x),\nabla\phi(x),u(x),x) \leq 0.\]
We will often say that $u \in \usc(\O)$ is a viscosity solution of $H \leq 0$ in $\O$ when $u$ is a viscosity subsolution of \eqref{eq:hjbgen2}.  

Similarly, we say that $u \in \lsc(\O)$ is a \emph{viscosity supersolution} of \eqref{eq:hjbgen2} if for every $x \in \O$ and every $\phi \in C^\infty(\R^n)$ such that $u - \phi$ has a local minimum at $x$ with respect to $\O$
\[H(\nabla ^2\phi(x),\nabla \phi(x),u(x),x) \geq 0.\]
We also say that $u \in \lsc(\O)$ is a viscosity solution of $H \geq 0$ in $\O$ when $u$ is a viscosity supersolution of \eqref{eq:hjbgen2}.

Finally, we say $u$ is \emph{viscosity solution} of \eqref{eq:hjbgen2} if $u$ is both a viscosity subsolution and a viscosity supersolution.  
\end{definition}

For more details on the rich theory of viscosity solutions, we refer the reader to the user's guide \cite{crandall1992user} and \cite{calderviscosity}.

\bibliography{main}
\bibliographystyle{abbrv}

\end{document}